\documentclass[12pt]{amsart}
\usepackage[latin1]{inputenc}
\usepackage{amsbsy}
\usepackage{amscd} 
\usepackage{amsfonts}
\usepackage{amsgen} 
\usepackage{amsmath}
\usepackage{amsopn} 
\usepackage{amssymb}
\usepackage{amstext}
\usepackage{amsthm} 
\usepackage{amsxtra}
\usepackage{color}
\usepackage{xcolor}

\theoremstyle{plain} 
\newtheorem{thm}{Theorem}[section]
\newtheorem{thmm}{Theorem}
\newtheorem{prop}[thm]{Proposition}
\newtheorem{lem}[thm]{Lemma}
\newtheorem{cor}[thm]{Corollary}
\newtheorem{quest}{Question}
\theoremstyle{definition}
\newtheorem{defn}[thm]{Definition}
\newtheorem{rem}[thm]{Remark}
\newtheorem{ex}[thm]{Example}

\numberwithin{equation}{section}

\renewcommand{\theta}{\vartheta}
\renewcommand{\phi}{\varphi}
\renewcommand{\epsilon}{\varepsilon}
\renewcommand{\subset}{\subseteq}
\renewcommand{\supset}{\supseteq}

\newcommand{\N}{\mathbb N}

\newcommand{\C}{\mathbb C}

\usepackage{calc}
\newcounter{PartitionDepth}
\newcounter{PartitionLength}

\newcommand{\partii}[3]{
 \begin{picture}(#3,#1)
 \setcounter{PartitionLength}{#3-#2}
 \setcounter{PartitionDepth}{-1-#1}
 \put(#2,\thePartitionDepth){\line(0,1){#1}}     
 \put(#3,\thePartitionDepth){\line(0,1){#1}}
 \put(#2,\thePartitionDepth){\line(1,0){\thePartitionLength}}
 \end{picture}}

\newcommand{\upparti}[2]{
 \begin{picture}(#2,#1)
 \setcounter{PartitionDepth}{#1}
 \put(#2,0){\line(0,1){#1}}
 \end{picture}}
\newcommand{\uppartii}[3]{
 \begin{picture}(#3,#1)
 \setcounter{PartitionLength}{#3-#2}
 \setcounter{PartitionDepth}{#1}
 \put(#2,0){\line(0,1){#1}}     
 \put(#3,0){\line(0,1){#1}}
 \put(#2,\thePartitionDepth){\line(1,0){\thePartitionLength}}
 \end{picture}}

\newsavebox{\boxpaarpart}
   \savebox{\boxpaarpart}
   { \begin{picture}(0.7,0.5)
     \put(0,0){\line(0,1){0.4}}
     \put(0.5,0){\line(0,1){0.4}}
     \put(0,0.4){\line(1,0){0.5}}
     \end{picture}}
\newsavebox{\boxbaarpart}
   \savebox{\boxbaarpart}
   { \begin{picture}(0.7,0.5)
     \put(0,0){\line(0,1){0.4}}
     \put(0.5,0){\line(0,1){0.4}}
     \put(0,0){\line(1,0){0.5}}
     \end{picture}}
\newsavebox{\boxvierpart}
   \savebox{\boxvierpart}
   { \begin{picture}(1.4,0.5)
     \put(0,0){\line(0,1){0.4}}
     \put(0.4,0){\line(0,1){0.4}}
     \put(0.8,0){\line(0,1){0.4}}
     \put(1.2,0){\line(0,1){0.4}}
     \put(0,0.4){\line(1,0){1.2}}
     \end{picture}}
\newsavebox{\boxvierpartrot}
   \savebox{\boxvierpartrot}
   { \begin{picture}(0.5,0.8)
     \put(0,-0.2){\line(0,1){0.3}}
     \put(0.4,-0.2){\line(0,1){0.3}}
     \put(0,0.1){\line(1,0){0.4}}
     \put(0,0.4){\line(0,1){0.3}}
     \put(0.4,0.4){\line(0,1){0.3}}
     \put(0,0.4){\line(1,0){0.4}}
     \put(0.2,0.1){\line(0,1){0.3}}
     \end{picture}}
\newsavebox{\boxvierpartrota}
   \savebox{\boxvierpartrota}
   { \begin{picture}(0.5,0.8)
     \put(0,-0.2){\line(0,1){0.3}}
     \put(0.4,-0.2){\line(0,1){0.3}}
     \put(0,0.1){\line(1,0){0.4}}
     \put(0,0.4){\line(0,1){0.3}}
     \put(0.4,0.4){\line(0,1){0.3}}
     \put(0,0.4){\line(1,0){0.4}}
     \put(0.1,0.1){\line(0,1){0.3}}
     \end{picture}}
\newsavebox{\boxvierpartrotb}
   \savebox{\boxvierpartrotb}
   { \begin{picture}(0.5,0.8)
     \put(0,-0.2){\line(0,1){0.3}}
     \put(0.4,-0.2){\line(0,1){0.3}}
     \put(0,0.1){\line(1,0){0.4}}
     \put(0,0.4){\line(0,1){0.3}}
     \put(0.4,0.4){\line(0,1){0.3}}
     \put(0,0.4){\line(1,0){0.4}}
     \put(0.3,0.1){\line(0,1){0.3}}
     \end{picture}}          
\newsavebox{\boxcrosspart}
   \savebox{\boxcrosspart}
   { \begin{picture}(0.5,0.8)
     \put(0,-0.2){\line(1,2){0.4}}
     \put(0.4,-0.2){\line(-1,2){0.4}}
     \end{picture}}
\newsavebox{\boxidpart}
   \savebox{\boxidpart}
   { \begin{picture}(0.3,0.9)
     \put(0,-0.2){\line(0,1){0.9}}
     \end{picture}}     
\newsavebox{\boxsingsingpart}
   \savebox{\boxsingsingpart}
   { \begin{picture}(0.3,0.9)
     \put(0,-0.2){\line(0,1){0.3}}
     \put(0,0.4){\line(0,1){0.3}}     
     \end{picture}}     
\newsavebox{\boxhalflibpart}
   \savebox{\boxhalflibpart}
   { \begin{picture}(1,0.8)
     \put(0,-0.2){\line(1,1){0.8}}
     \put(0.8,-0.2){\line(-1,1){0.8}}
     \put(0.4,-0.2){\line(0,1){0.8}}
     \end{picture}}   
\newsavebox{\boxpaarpaarpartrot}
   \savebox{\boxpaarpaarpartrot}
   { \begin{picture}(0.5,0.8)
     \put(0,-0.2){\line(0,1){0.3}}
     \put(0.4,-0.2){\line(0,1){0.3}}
     \put(0,0.1){\line(1,0){0.4}}
     \put(0,0.4){\line(0,1){0.3}}
     \put(0.4,0.4){\line(0,1){0.3}}
     \put(0,0.4){\line(1,0){0.4}}
     \end{picture}}     

\newsavebox{\boxcircle}
   \savebox{\boxcircle}
   { \begin{picture}(0.2,0.2)
     \put(0.1,0.1){\color{gray}\circle{0.2}}
     \end{picture}}     
\newsavebox{\boxsquare}
   \savebox{\boxsquare}
   { \begin{picture}(0.2,0.2)
     \put(0,0){\color{gray}\line(1,0){0.15}}
     \put(0,0.15){\color{gray}\line(1,0){0.15}}
     \put(0,0){\color{gray}\line(0,1){0.15}}
     \put(0.15,0){\color{gray}\line(0,1){0.15}}          
     \end{picture}}     
\newsavebox{\boxpoints}
   \savebox{\boxpoints}
   { \begin{picture}(0.3,0.3)
     \put(-0.35,-0.4){\usebox{\boxcircle}}
     \put(-0.02,0.02){\usebox{\boxcircle}}
     \end{picture}}     
\newsavebox{\boxFourOneIdTwo} 
   \savebox{\boxFourOneIdTwo}
   { \begin{picture}(1.2,1.5)
     \put(0,0){\usebox{\boxpoints}}
     \put(0.4,0){\usebox{\boxpoints}}
     \put(0,1.1){\usebox{\boxpoints}}
     \put(0.4,1.1){\usebox{\boxpoints}}
     \put(0,0){\usebox{\boxvierpartrota}}
      \put(0.03,0){\usebox{\boxvierpartrota}}     
     \put(0.3,0.4){\usebox{\boxidpart}}
      \put(0.33,0.4){\usebox{\boxidpart}}     
     \put(0.7,0.4){\usebox{\boxidpart}}
      \put(0.73,0.4){\usebox{\boxidpart}}      
     \end{picture}}     
\newsavebox{\boxIdOneFourTwo} 
   \savebox{\boxIdOneFourTwo}
   { \begin{picture}(1.2,1.5)
     \put(0,0){\usebox{\boxpoints}}
     \put(0.4,0){\usebox{\boxpoints}}
     \put(0,1.1){\usebox{\boxpoints}}
     \put(0.4,1.1){\usebox{\boxpoints}}
     \put(0.3,0.4){\usebox{\boxvierpartrotb}}
      \put(0.33,0.4){\usebox{\boxvierpartrotb}}     
     \put(0,0){\usebox{\boxidpart}}
      \put(0.03,0){\usebox{\boxidpart}}     
     \put(0.4,0){\usebox{\boxidpart}}
      \put(0.43,0){\usebox{\boxidpart}}      
     \end{picture}}             
\newsavebox{\boxFourOneFourTwo} 
   \savebox{\boxFourOneFourTwo}
   { \begin{picture}(1.2,1.5)
     \put(0,0){\usebox{\boxpoints}}
     \put(0.4,0){\usebox{\boxpoints}}
     \put(0,1.1){\usebox{\boxpoints}}
     \put(0.4,1.1){\usebox{\boxpoints}}
     \put(0,0){\usebox{\boxvierpartrot}}
      \put(0.03,0){\usebox{\boxvierpartrot}}     
     \put(0.3,0.4){\usebox{\boxvierpartrot}}
      \put(0.33,0.4){\usebox{\boxvierpartrot}}     
     \end{picture}}   
\newsavebox{\boxSingSingOneIdTwo} 
   \savebox{\boxSingSingOneIdTwo}
   { \begin{picture}(1,1.5)
     \put(0,0){\usebox{\boxpoints}}
     \put(0,1.1){\usebox{\boxpoints}}
     \put(0,0){\usebox{\boxsingsingpart}}
      \put(0.03,0){\usebox{\boxsingsingpart}}     
     \put(0.3,0.4){\usebox{\boxidpart}}
      \put(0.33,0.4){\usebox{\boxidpart}}     
     \end{picture}}        
\newsavebox{\boxIdOneSingSingTwo} 
   \savebox{\boxIdOneSingSingTwo}
   { \begin{picture}(1,1.5)
     \put(0,0){\usebox{\boxpoints}}
     \put(0,1.1){\usebox{\boxpoints}}
     \put(0,0){\usebox{\boxidpart}}
      \put(0.03,0){\usebox{\boxidpart}}
     \put(0.3,0.4){\usebox{\boxsingsingpart}}
      \put(0.33,0.4){\usebox{\boxsingsingpart}}
     \end{picture}}  
\newsavebox{\boxSingSingOneSingSingTwo} 
   \savebox{\boxSingSingOneSingSingTwo}
   { \begin{picture}(1.2,1.5)
     \put(0,0){\usebox{\boxpoints}}
     \put(0,1.1){\usebox{\boxpoints}}
     \put(0,0){\usebox{\boxsingsingpart}}
      \put(0.03,0){\usebox{\boxsingsingpart}}     
     \put(0.3,0.4){\usebox{\boxsingsingpart}}
      \put(0.33,0.4){\usebox{\boxsingsingpart}}
     \end{picture}}            
\newsavebox{\boxSingOneSingTwo} 
   \savebox{\boxSingOneSingTwo}
   { \begin{picture}(1.2,1.5)
     \put(0,0){\usebox{\boxpoints}}
     \put(0.1,-0.1){$\uparrow$}
     \put(0.4,0.3){$\uparrow$}
     \end{picture}}  
\newsavebox{\boxCrossOneIdTwo} 
   \savebox{\boxCrossOneIdTwo}
   { \begin{picture}(1.2,1.5)
     \put(0,0){\usebox{\boxpoints}}
     \put(0.4,0){\usebox{\boxpoints}}
     \put(0,1.1){\usebox{\boxpoints}}
     \put(0.4,1.1){\usebox{\boxpoints}}     
     \put(0,0){\usebox{\boxcrosspart}}
      \put(0.03,0){\usebox{\boxcrosspart}}     
     \put(0.3,0.4){\usebox{\boxidpart}}
      \put(0.33,0.4){\usebox{\boxidpart}}
     \put(0.7,0.4){\usebox{\boxidpart}}
      \put(0.73,0.4){\usebox{\boxidpart}}
     \end{picture}}               
\newsavebox{\boxIdOneCrossTwo} 
   \savebox{\boxIdOneCrossTwo} 
   { \begin{picture}(1.2,1.5)
     \put(0,0){\usebox{\boxpoints}}
     \put(0.4,0){\usebox{\boxpoints}}
     \put(0,1.1){\usebox{\boxpoints}}
     \put(0.4,1.1){\usebox{\boxpoints}}     
     \put(0,0){\usebox{\boxidpart}}
      \put(0.03,0){\usebox{\boxidpart}}     
     \put(0.4,0){\usebox{\boxidpart}}
      \put(0.43,0){\usebox{\boxidpart}}
     \put(0.3,0.4){\usebox{\boxcrosspart}}
      \put(0.33,0.4){\usebox{\boxcrosspart}}
     \end{picture}}  
\newsavebox{\boxCrossOneCrossTwo} 
   \savebox{\boxCrossOneCrossTwo}
   { \begin{picture}(1.2,1.5)
     \put(0,0){\usebox{\boxpoints}}
     \put(0.4,0){\usebox{\boxpoints}}
     \put(0,1.1){\usebox{\boxpoints}}
     \put(0.4,1.1){\usebox{\boxpoints}}     
     \put(0,0){\usebox{\boxcrosspart}}
      \put(0.03,0){\usebox{\boxcrosspart}}     
     \put(0.3,0.4){\usebox{\boxcrosspart}}
      \put(0.33,0.4){\usebox{\boxcrosspart}}
     \end{picture}}  
\newsavebox{\boxvierschraegpart}
   \savebox{\boxvierschraegpart}
   { \begin{picture}(0.5,0.8)
     \put(0,-0.2){\line(0,1){0.3}}
     \put(0.3,0.2){\line(0,1){0.3}}
     \put(0,0.4){\line(0,1){0.3}}
     \put(0.3,0.8){\line(0,1){0.3}}
     \put(0.15,0.35){\line(0,1){0.3}}
     \put(0,0.1){\line(3,4){0.75}}
     \put(0,0.4){\line(3,4){0.75}}
     \end{picture}}       
\newsavebox{\boxpaarpaarschraegpart}
   \savebox{\boxpaarpaarschraegpart}
   { \begin{picture}(0.5,0.8)
     \put(0,-0.2){\line(0,1){0.3}}
     \put(0.3,0.2){\line(0,1){0.3}}
     \put(0,0.4){\line(0,1){0.3}}
     \put(0.3,0.8){\line(0,1){0.3}}
     \put(0,0.1){\line(3,4){0.75}}
     \put(0,0.4){\line(3,4){0.75}}
     \end{picture}}       
\newsavebox{\boxzweischraegpart}
   \savebox{\boxzweischraegpart}
   { \begin{picture}(0.5,0.8)
     \put(0,-0.2){\line(0,1){0.3}}
     \put(0.3,0.2){\line(0,1){0.3}}
     \put(0,0.1){\line(3,4){0.75}}
     \end{picture}}       
\newsavebox{\boxFourOneTwo} 
   \savebox{\boxFourOneTwo}
   { \begin{picture}(1.2,1.5)
     \put(0,0){\usebox{\boxpoints}}
     \put(0,1.1){\usebox{\boxpoints}}
     \put(0,0){\usebox{\boxvierschraegpart}}
      \put(0.03,0){\usebox{\boxvierschraegpart}}    
     \put(0.85,0.6){{\color{white}\circle*{0.5}}}
     \put(0.85,0.8){{\color{white}\circle*{0.5}}}
     \put(0.9,1){{\color{white}\circle*{0.4}}}     
     \put(0.95,1.3){{\color{white}\circle*{0.3}}}          
     \end{picture}}      
\newsavebox{\boxPaarOneTwo} 
   \savebox{\boxPaarOneTwo}
   { \begin{picture}(1.2,1.5)
     \put(0,0){\usebox{\boxpoints}}
     \put(0,0){\usebox{\boxzweischraegpart}}
      \put(0.03,0){\usebox{\boxzweischraegpart}}    
     \put(0.85,0.6){{\color{white}\circle*{0.5}}}
     \put(0.85,0.8){{\color{white}\circle*{0.5}}}
     \put(0.9,1){{\color{white}\circle*{0.4}}}     
     \put(0.95,1.3){{\color{white}\circle*{0.3}}}          
     \end{picture}}  
\newsavebox{\boxProzent} 
   \savebox{\boxProzent}
   { \begin{picture}(1,1.5)
     \put(0,0){\usebox{\boxpoints}}
     \put(0,1.1){\usebox{\boxpoints}}
     \put(0.3,0.5){\line(0,1){0.2}}
      \put(0.33,0.5){\line(0,1){0.2}}
     \put(0.6,0.2){\line(0,1){0.2}}
      \put(0.63,0.2){\line(0,1){0.2}}
     \put(0.3,-0.15){\line(1,4){0.3}}
      \put(0.33,-0.15){\line(1,4){0.3}}
     \end{picture}}      
\newsavebox{\boxPaarPaarOneIdTwo} 
   \savebox{\boxPaarPaarOneIdTwo}
   { \begin{picture}(1.2,1.5)
     \put(0,0){\usebox{\boxpoints}}
     \put(0.4,0){\usebox{\boxpoints}}
     \put(0,1.1){\usebox{\boxpoints}}
     \put(0.4,1.1){\usebox{\boxpoints}}
     \put(0,0){\usebox{\boxpaarpaarpartrot}}
      \put(0.03,0){\usebox{\boxpaarpaarpartrot}}     
     \put(0.3,0.4){\usebox{\boxidpart}}
      \put(0.33,0.4){\usebox{\boxidpart}}     
     \put(0.7,0.4){\usebox{\boxidpart}}
      \put(0.73,0.4){\usebox{\boxidpart}}      
     \end{picture}}     
\newsavebox{\boxIdOnePaarPaarTwo} 
   \savebox{\boxIdOnePaarPaarTwo}
   { \begin{picture}(1.2,1.5)
     \put(0,0){\usebox{\boxpoints}}
     \put(0.4,0){\usebox{\boxpoints}}
     \put(0,1.1){\usebox{\boxpoints}}
     \put(0.4,1.1){\usebox{\boxpoints}}
     \put(0.3,0.4){\usebox{\boxpaarpaarpartrot}}
      \put(0.33,0.4){\usebox{\boxpaarpaarpartrot}}     
     \put(0,0){\usebox{\boxidpart}}
      \put(0.03,0){\usebox{\boxidpart}}     
     \put(0.4,0){\usebox{\boxidpart}}
      \put(0.43,0){\usebox{\boxidpart}}      
     \end{picture}}                        
\newsavebox{\boxCrossOneTwo} 
   \savebox{\boxCrossOneTwo}
   { \begin{picture}(1,1.5)
     \put(0.6,0.2){\line(-1,2){0.4}}
      \put(0.63,0.2){\line(-1,2){0.4}}
     \put(0.3,-0.15){\line(1,4){0.3}}
      \put(0.33,-0.15){\line(1,4){0.3}}
     \put(0.2,0.8){{\color{white}\circle*{0.4}}}      
     \put(0,0){\usebox{\boxpoints}}
     \put(0,1.1){\usebox{\boxpoints}}
     \end{picture}}     
\newsavebox{\boxPaarPaarOneTwo} 
   \savebox{\boxPaarPaarOneTwo}
   { \begin{picture}(1.2,1.5)
     \put(0,0){\usebox{\boxpoints}}
     \put(0,1.1){\usebox{\boxpoints}}
     \put(0,0){\usebox{\boxpaarpaarschraegpart}}
      \put(0.03,0){\usebox{\boxpaarpaarschraegpart}}    
     \put(0.85,0.6){{\color{white}\circle*{0.5}}}
     \put(0.85,0.8){{\color{white}\circle*{0.5}}}
     \put(0.9,1){{\color{white}\circle*{0.4}}}     
     \put(0.95,1.3){{\color{white}\circle*{0.3}}}          
     \end{picture}}            
\newsavebox{\boxHalfThree} 
   \savebox{\boxHalfThree}
   { \begin{picture}(1.2,1.5)
     \put(0,0){\usebox{\boxpoints}}
     \put(0,1.1){\usebox{\boxpoints}}
     \put(0.4,1.1){\usebox{\boxpoints}}
     \put(0,0){\usebox{\boxidpart}}
      \put(0.03,0){\usebox{\boxidpart}}     
     \put(0.7,0.3){\line(0,1){0.4}}
      \put(0.73,0.3){\line(0,1){0.4}}
     \put(0.3,0.3){\line(1,0){0.4}}
      \put(0.3,0.33){\line(1,0){0.4}}     
     \put(0.3,0.4){\usebox{\boxidpart}}
      \put(0.33,0.4){\usebox{\boxidpart}} 
     \put(1,0.7){\line(0,1){0.4}}
      \put(1.03,0.7){\line(0,1){0.4}}
     \end{picture}}         
\newsavebox{\boxGrauKreuz} 
   \savebox{\boxGrauKreuz}
   { \begin{picture}(1,1)
     \put(-0.3,0.2){{\color{lightgray}\line(1,0){1}}}
     \put(-0.1,-0.2){{\color{lightgray}\line(5,6){0.6}}}
     \put(0,0){$\circ$}
     \end{picture}}   
\newsavebox{\boxGrauKreuzg} 
   \savebox{\boxGrauKreuzg}
   { \begin{picture}(1,1)
     \put(-0.3,0.2){{\color{lightgray}\line(1,0){1}}}
     \put(-0.1,-0.2){{\color{lightgray}\line(5,6){0.6}}}
     \put(0,0){{\color{gray}$\bullet$}}
     \end{picture}}      
\newsavebox{\boxGrauKreuzs} 
   \savebox{\boxGrauKreuzs}
   { \begin{picture}(1,1)
     \put(-0.3,0.2){{\color{lightgray}\line(1,0){1}}}
     \put(-0.1,-0.2){{\color{lightgray}\line(5,6){0.6}}}
     \put(0,0){$\bullet$}
     \end{picture}}   
\newsavebox{\boxGrLL} 
   \savebox{\boxGrLL}
   { \begin{picture}(2,2)
     \put(-0.75,0){\usebox{\boxGrauKreuz}}
     \put(-0.15,0.7){\usebox{\boxGrauKreuz}}
     \end{picture}}      
\newsavebox{\boxGrLLgg} 
   \savebox{\boxGrLLgg}
   { \begin{picture}(2,2)
     \put(-0.75,0){\usebox{\boxGrauKreuzg}}
     \put(-0.15,0.7){\usebox{\boxGrauKreuzg}}
     \end{picture}}      
\newsavebox{\boxGrLLgs} 
   \savebox{\boxGrLLgs}
   { \begin{picture}(2,2)
     \put(-0.75,0){\usebox{\boxGrauKreuzg}}
     \put(-0.15,0.7){\usebox{\boxGrauKreuzs}}
     \end{picture}}      
\newsavebox{\boxGrLLsg} 
   \savebox{\boxGrLLsg}
   { \begin{picture}(2,2)
     \put(-0.75,0){\usebox{\boxGrauKreuzs}}
     \put(-0.15,0.7){\usebox{\boxGrauKreuzg}}
     \end{picture}}      
\newsavebox{\boxGrLLL} 
   \savebox{\boxGrLLL}
   { \begin{picture}(2,2)
     \put(-0.75,-0.2){\usebox{\boxGrauKreuz}}
     \put(-0.15,0.5){\usebox{\boxGrauKreuz}}
     \put(0.45,1.2){\usebox{\boxGrauKreuz}}     
     \end{picture}}      
\newsavebox{\boxGrLLLL} 
   \savebox{\boxGrLLLL}
   { \begin{picture}(3,3)
     \put(-0.75,0){\usebox{\boxGrauKreuz}}
     \put(-0.15,0.7){\usebox{\boxGrauKreuz}}
     \put(0.45,1.4){\usebox{\boxGrauKreuz}}  
     \put(1.05,2.1){\usebox{\boxGrauKreuz}}  
     \end{picture}}    
\newsavebox{\boxidpartww}
   \savebox{\boxidpartww}
   { \begin{picture}(0.5,1.2)
     \put(0.2,0.15){\line(0,1){0.7}}
     \put(0,-0.2){$\circ$}
     \put(0,0.8){$\circ$}
     \end{picture}}
\newsavebox{\boxidpartbb}
   \savebox{\boxidpartbb}
   { \begin{picture}(0.5,1.2)
     \put(0.2,0.15){\line(0,1){0.7}}
     \put(0,-0.2){$\bullet$}
     \put(0,0.8){$\bullet$}
     \end{picture}}     
\newsavebox{\boxpaarpartwb}
   \savebox{\boxpaarpartwb}
   { \begin{picture}(1,1)
     \put(0.2,0.2){\line(0,1){0.4}}
     \put(0.7,0.2){\line(0,1){0.4}}
     \put(0.2,0.6){\line(1,0){0.5}}
     \put(0.5,-0.2){$\bullet$}
     \put(0,-0.2){$\circ$}
     \end{picture}}     
\newsavebox{\boxpaarpartbw}
   \savebox{\boxpaarpartbw}
   { \begin{picture}(1,1)
     \put(0.2,0.2){\line(0,1){0.4}}
     \put(0.7,0.2){\line(0,1){0.4}}
     \put(0.2,0.6){\line(1,0){0.5}}
     \put(0.5,-0.2){$\circ$}
     \put(0,-0.2){$\bullet$}
     \end{picture}}

\newcommand{\paarpart}{\usebox{\boxpaarpart}}
\newcommand{\baarpart}{\usebox{\boxbaarpart}}

\newcommand{\halflibpart}{\usebox{\boxhalflibpart}}
\newcommand{\vierpartrot}{\usebox{\boxvierpartrot}}

\newcommand{\crosspart}{\usebox{\boxcrosspart}}

\newcommand{\idpart}{|}
\newcommand{\singleton}{\uparrow}
\newcommand{\downsingleton}{\downarrow}

\newcommand{\FourOneIdTwopart}{\usebox{\boxFourOneIdTwo}}
\newcommand{\IdOneFourTwopart}{\usebox{\boxIdOneFourTwo}}

\newcommand{\SingSingOneIdTwopart}{\usebox{\boxSingSingOneIdTwo}}
\newcommand{\IdOneSingSingTwopart}{\usebox{\boxIdOneSingSingTwo}}

\newcommand{\SingOneSingTwopart}{\usebox{\boxSingOneSingTwo}}
\newcommand{\CrossOneIdTwopart}{\usebox{\boxCrossOneIdTwo}}
\newcommand{\IdOneCrossTwopart}{\usebox{\boxIdOneCrossTwo}}
\newcommand{\CrossOneCrossTwopart}{\usebox{\boxCrossOneCrossTwo}}
\newcommand{\FourOneTwopart}{\usebox{\boxFourOneTwo}}
\newcommand{\PaarOneTwopart}{\usebox{\boxPaarOneTwo}}

\newcommand{\PaarPaarOneIdTwopart}{\usebox{\boxPaarPaarOneIdTwo}}
\newcommand{\IdOnePaarPaarTwopart}{\usebox{\boxIdOnePaarPaarTwo}}
\newcommand{\CrossOneTwopart}{\usebox{\boxCrossOneTwo}}
\newcommand{\PaarPaarOneTwopart}{\usebox{\boxPaarPaarOneTwo}}
\newcommand{\HalfThreepart}{\usebox{\boxHalfThree}}
\newcommand{\paarpartwb}{\usebox{\boxpaarpartwb}}
\newcommand{\paarpartbw}{\usebox{\boxpaarpartbw}}
\newcommand{\idpartww}{\usebox{\boxidpartww}}
\newcommand{\idpartbb}{\usebox{\boxidpartbb}}

\newcommand{\CC}{\mathcal C}
\newcommand{\Pm}{P^{(m)}}

\DeclareMathOperator{\id}{id}
\DeclareMathOperator{\Hom}{Hom}
\DeclareMathOperator{\lspan}{span}
\DeclareMathOperator{\symm}{symm}
\DeclareMathOperator{\resplevels}{resplevels}
\DeclareMathOperator{\even}{even}
\DeclareMathOperator{\nodiagonal}{nodiagonal}
\DeclareMathOperator{\noviceversa}{noviceversa}
\DeclareMathOperator{\nc}{nc}

\begin{document}
\title{Quantum groups based on spatial partitions}
\author{Guillaume C\'ebron}
\address{Universit\'e Paul Sabatier, Institut de Math\'ematiques de Toulouse,
118 Route de Narbonne, 31062 Toulouse, France}
\author{Moritz Weber}
\address{Saarland University, Fachbereich Mathematik, Postfach 151150,
66041 Saarbr\"ucken, Germany}
\email{guillaume.cebron@math.univ-toulouse.fr, weber@math.uni-sb.de}
\date{\today}
\subjclass[2010]{20G42 (Primary); 05A18, 05E10, 46L54 (Secondary)}
\keywords{set partitions, three-dimensional partitions, spatial partitions, compact matrix quantum groups, easy quantum groups, partition quantum groups, Banica-Speicher quantum groups, free orthogonal quantum groups, tensor categories, Kronecker product}

\begin{abstract}
We define new compact matrix quantum groups whose intertwiner spaces are dual to tensor categories of three-dimensional set partitions -- which we call \emph{spatial partitions}. This extends substantially Banica and Speicher's approach of the so called easy quantum groups: It enables us to find new examples of quantum subgroups of Wang's free orthogonal quantum group $O_n^+$ which do \emph{not} contain the symmetric group $S_n$; we may define new kinds of products of quantum groups coming from new products of categories of partitions; and we give a quantum group interpretation of certain categories of partitions which do neither contain the pair partition nor the identity partition.
\end{abstract}

\maketitle

\section*{Introduction}

Compact matrix quantum groups have been defined by Woronowicz in the 1980's \cite{WoCMQG}. In the 1990's, Wang \cite{WangOrth} gave a definition of a free quantum version $O_n^+$ of the orthogonal group $O_n\subset M_n(\C)$. The idea is basically to replace the scalar entries $u_{ij}$ of an orthogonal matrix by noncommuting variables. One can think of the $u_{ij}$ as operators on a Hilbert space, for instance. The quantum group $O_n^+$ contains the group $O_n$, hence there are somehow more orthogonal rotations in the quantum world than in the classical world. 

In order to understand quantum subgroups of $O_n^+$, Banica and Speicher \cite{BS} developped the theory of easy quantum groups, which we also call Banica-Speicher quantum groups throughout the article. They are based on set partitions which are decompositions of finite ordered sets into disjoint subsets. In a Tannaka-Krein (or Schur-Weyl) sense, the intertwiner spaces of Banica-Speicher quantum groups are dual to  categories of partitions \cite{WoTK, BS, TW}. More precisely, to each partition $p$  we associate a linear map $T_p$. A category of partitions is a set of partitions which is closed under taking tensor products, composition and involution of partitions. These operations on partitions $p$ correspond exactly to canonical operations on the linear maps $T_p$ turning the linear span of these $T_p$ into a tensor category. A quantum subgroup $G\subset O_n^+$ of $O_n^+$ is called easy or a Banica-Speicher quantum group \cite{BS}, if its intertwiners are given by such a linear span of maps $T_p$ indexed by partitions $p$ coming from a category of partitions. Hence, Banica-Speicher quantum groups (operator algebraic objects) are in one-to-one correspondence to categories of partitions (combinatorial objects).

The motivation for our article came from the following three questions. 

Firstly, any category of partitions is required to contain two particular partitions as a base case: the pair partition $\paarpart$ and the identity partition $\idpart$ (in order to obtain a quantum subgroup of $O_n^+$). 

\begin{quest}\label{Q1}
Can we replace these base partitions by other base partitions and still associate quantum groups to such categories? 
\end{quest}

From a combinatorial point of view, there is no problem in studying categories of partitions with different base cases, but until now, there was no interpretation of such objects on the quantum group side. 

Secondly, given two categories of partitions $\CC_1$ and $\CC_2$. 

\begin{quest}\label{Q2}
Can we form a new category of partitions out of two given ones by some product construction which resembles product constructions on the level of quantum groups?
\end{quest}

Thirdly, the approach to construct quantum subgroups $G$ of $O_n^+$ via Banica-Speicher quantum groups comes with the restriction that $G$ contains the symmetric group $S_n$.

\begin{quest}\label{Q3}
How can we extend the machinery of Banica-Speicher quantum groups in order to cover quantum groups $S_n\not\subset G\subset O_n^+$?
\end{quest}

Surprisingly, we can give answers to all three questions at the same time, with our new machinery. On the way, we define new products of general quantum subgroups of $O_n^+$ and we find many new examples of quantum subgroups of $O_n^+$.

Banica-Speicher quantum groups have links to Voiculescu's free probability theory \cite{NS, VSW}, for instance via de Finetti theorems \cite{SpK, BCS}. See also \cite{TW, RW, RWd, FW, Web, LT, F, Br} as an incomplete list for recent work on Banica-Speicher quantum groups or on $O_n^+$. Question \ref{Q3} has also been tackled in the very recent preprint by Speicher and the second author \cite{SW}.

\pagebreak

\section{Main ideas and main results}

The key point of Banica and Speicher's approach is to consider a partition  $p\in P(k,l)$ of a set with  $k+l$ elements ($k$ ``upper'' ones and $l$ ``lower'' ones) and to associate a linear map $T_p:(\C^n)^{\otimes k}\to(\C^n)^{\otimes l}$ to it, for a fixed $n\in\N$. If the number $n$ can be written as a product $n=n_1\cdots n_m$ for $n_i\in\N$, we obtain
\begin{tabular}{p{8.7cm}p{2.2cm}p{3cm}}
&\\
$T_p:(\C^{n_1\cdots n_m})^{\otimes k}\to(\C^{n_1\cdots n_m})^{\otimes l}$ &governed by &$ p\in P(k,l)$.\\
&
\end{tabular}

Our main tool is derived from the following simple observation. If we consider partitions in $P(km,lm)$ and apply the assignment $p\mapsto T_p$, we obtain a map
\begin{tabular}{p{8.7cm}p{2.2cm}p{3cm}}
&\\
$T_p:(\C^{n_1}\otimes\cdots \otimes\C^{n_m})^{\otimes k}\to(\C^{n_1}\otimes\cdots \otimes\C^{n_m})^{\otimes l}$ &governed by &$p\in P(km,lm)$.\\
&
\end{tabular}

Under the isomorphism $\C^n=\C^{n_1\cdots n_m}\cong \C^{n_1}\otimes\ldots\otimes\C^{n_m},$ this enables us to find many more maps from $(\C^n)^{\otimes k}$ to $(\C^n)^{\otimes l}$ compared to Banica and Speicher's approach, since we may use partitions on more points (Section \ref{SectLinMMaps}).

On a technical level, it is convenient to view partitions in $P(km,lm)$ as three-dimensional partitions (on $k\times m$ ``upper'' points and $l\times m$ ``lower'' points) and to speak about the set $\Pm(k,l)$ of spatial partitions (see Section \ref{SectAdapted}). Then, spatial partition quantum groups  are defined as quantum subgroups of $O_{n_1\cdots n_m}^+$ whose intertwiner spaces are given by maps indexed by spatial partitions (Section \ref{SectMeasy}). The sets of spatial partitions which corresponds to categories of intertwiner spaces will be called categories of  spatial partitions: they are sets of spatial partitions which are closed under tensor product, composition and involution, and which contains the base partitions 
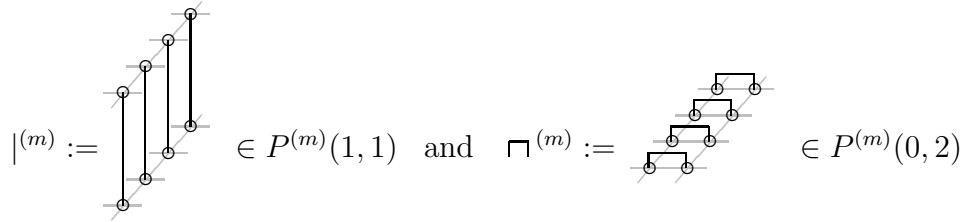
\setlength{\unitlength}{0.5cm}
\begin{center}
\begin{picture}(25,7)
\put(0,2){$\idpart^{(m)}:=$}
\put(3,0.5){\usebox{\boxGrLLLL}}
\put(3,3.5){\usebox{\boxGrLLLL}}
\put(3,0.7){\line(0,1){3}}
\put(3.6,1.4){\line(0,1){3}}
\put(4.2,2.1){\line(0,1){3}}
\put(4.8,2.8){\line(0,1){3}}
\put(6,2){$\in P^{(m)}(1,1)$}
\put(11,2){and}
\put(13,2){$\paarpart^{(m)}:=$}
\put(17,1.5){\usebox{\boxGrLLLL}}
\put(18,1.5){\usebox{\boxGrLLLL}}
\newsavebox{\BoxPa}
   \savebox{\BoxPa}
   { \begin{picture}(1,1)
   \put(0,0){\line(0,1){0.4}}
   \put(1,0){\line(0,1){0.4}}
   \put(0,0.4){\line(1,0){1}}
   \end{picture}}
\put(16.7,1.7){\usebox{\BoxPa}}
\put(17.3,2.4){\usebox{\BoxPa}}
\put(17.9,3.1){\usebox{\BoxPa}}
\put(18.5,3.8){\usebox{\BoxPa}}
\put(21,2){$\in P^{(m)}(0,2)$}
\end{picture}
\end{center}
(see Section \ref{SectMCateg}). Note that we do not require the containment of $\idpart\in P(1,1)$ and $\paarpart\in P(0,2)$, on the contrary to the categories of partitions of Banica and Speicher. This answers Question \ref{Q1} (see also Remark \ref{RemBaseCase}). If the intertwiner spaces of a quantum group $G$ is given by the linear spans of maps $T_p$ indexed by partitions $p$ coming from a category $\CC$ of spatial partitions, we usually say that $\CC$ corresponds to the quantum group $G$.

In order to prepare an answer to Question \ref{Q2} observe that given  two categories $\CC_i\subset P$ for $i=1,2$, we may form the category $\CC_1\times \CC_2\subset P^{(2)}$ by placing partitions from $\CC_1$ on the first level and partitions from $\CC_2$ on the second one in our three-dimensional picture. On the other hand, given two compact matrix quantum groups $(G,u)$ and $(H,v)$ such that the matrices $u$ and $v$ have the same size, we can form the glued direct product $G\tilde\times H$ of \cite[Def. 6.4]{TW}, and more generally, the glued direct product $G\tilde\times_p H$ with amalgamation over a partition $p\in P^{(2)}$ (see Definition~\ref{DefAmalgam}), given by 
\begin{align*}
C^*(u_{ij}v_{kl})\subset C(G)\otimes_{\textnormal{max}} C(H) / \langle u_{ij}v_{kl} \textnormal{ satisfy intertwiner relations associated to } p\rangle.
\end{align*}
We then have the following answer to Question \ref{Q2}.

\begin{thmm}[Thm. \ref{ThmProduct}, Thm. \ref{ThmAmalProd}] 
Let $(G_i,u_i)\subset O_{n}^+$ be Banica-Speicher quantum groups with categories $\CC_i\subset P$ for $i=1,2$. Then,
\begin{align*}
&\CC_1\times \CC_2 &&\textit{ corresponds to } && G_1\tilde\times G_2\subset O_{n^2}^+;\\
&\langle\CC_1\times \CC_2,p\rangle &&\textit{ corresponds to } && G_1\tilde\times_p G_2\subset O_{n^2}^+.
\end{align*}
\end{thmm}

Regarding Question \ref{Q3}, we have the following result.

\begin{thmm}[Thm. \ref{PropMin}]
For $n_1=\ldots =n_m=n$ the maximal category $\Pm$ of all spatial partitions corresponds to $S_n\subset O_{n^m}^+$. As a consequence, we have $S_n\subset G\subset O_{n^m}^+$ for all spatial partition quantum groups; in particular $S_{n^m}\not\subset G\subset O_{n^m}^+$ is possible. 
\end{thmm}

With our approach, we may find many new examples of quantum subgroups of $O_{n^2}^+$. The next two theorems of combinatorial type show that the step from $m=1$ to $m=2$ is huge.

\begin{thmm}[Thm. \ref{PropGenerators}, Cor. \ref{CorGenerators}, Thm. \ref{PropPZwei}]
The category $P^{(2)}$ (resp. $P_2^{(2)}$) of all spatial (resp. spatial pair) partitions is generated by the partitions $\idpart^{(2)}, \paarpart^{(2)}$ and
\begin{align*}
P^{(2)}: &\quad\PaarOneTwopart, \CrossOneIdTwopart, \IdOneCrossTwopart, \FourOneIdTwopart, \IdOneFourTwopart,   \singleton^{(2)};\\
P_2^{(2)}: & \quad \PaarOneTwopart,  \CrossOneIdTwopart, \IdOneCrossTwopart, \PaarPaarOneIdTwopart, \IdOnePaarPaarTwopart&(\textit{note that } \PaarPaarOneIdTwopart\neq\FourOneIdTwopart \textit{ and } \IdOnePaarPaarTwopart\neq\IdOneFourTwopart).
\end{align*}
\end{thmm}
Recall that in the case $m=1$, we have $P=\langle\crosspart,\vierpartrot,\singleton\rangle$ and $P_2=\langle\crosspart\rangle$ (see \cite{Web}).

\begin{thmm}[Thm. \ref{ThmEx}]
The following subcategories of $P_2^{(2)}$ are all distinct:
\[\langle\emptyset\rangle,  \langle\halflibpart^{(2)}\rangle,\langle \CrossOneCrossTwopart\rangle,\langle\PaarPaarOneTwopart\rangle,\langle\CrossOneTwopart\rangle,\langle\CrossOneTwopart,\PaarPaarOneTwopart\rangle,\langle\PaarOneTwopart\rangle,\langle\CrossOneTwopart,\PaarOneTwopart\rangle,P_2^{(2)}, \CC_1\times\CC_2 \]
with $\CC_i\in\{NC_2,\langle\halflibpart\rangle,P_2\}$ (non-exhaustive list).
\end{thmm}
Recall that in the case $m=1$, we have exactly three subcategories of $P_2$, namely $NC_2=\langle\emptyset\rangle,\langle\halflibpart\rangle$ and $P_2=\langle\crosspart\rangle$ (see \cite{Web}).

We end the article (Section \ref{SectOutlook}) with an outlook in the unitary case. In particular, we define a free product $\CC_1*\CC_2$ of two categories due to a certain noncrossing condition\footnote{Note that it is not so clear a priori how to define noncrossing three-dimensional partitions.} between the levels. 

\pagebreak

\section{The combinatorics: spatial partitions and categories}

Let us first introduce the combinatorics of our objects.

\subsection{Partitions} 
\label{SectPartitions}

Let $k,l\in\N_0=\{0,1,2,\ldots\}$ and consider the ordered set \linebreak$\{1,\ldots,k,k+1,\ldots,k+l\}$. A \emph{(set) partition} is a decomposition of this set into disjoint subsets, the \emph{blocks}. We usually speak of the points $1,\ldots,k$ as ``upper points'' while $k+1,\ldots,k+l$ are ``lower points''. We identify a partition with the picture placing $k$ points on an upper line, $l$ points on a lower line and connecting these points by strings according to the block pattern (where the upper points are numbered from left to right whereas the lower points are numbered from right to left). The set of all partitions with $k$ upper and $l$ lower points is denoted by $P(k,l)$ and we put $P:=\bigcup_{k,l\in\N_0} P(k,l)$.

\begin{ex}\label{Expq}
Let $k=4$ and $l=3$. The partitions 
\[p=\{\{1,2\},\{3,4,5\},\{6,7\}\} \qquad\textnormal{and}\qquad
q=\{\{1,6\},\{2,7\},\{3,4\},\{5\}\}\]
in $P(4,3)$ are represented by the following pictures.

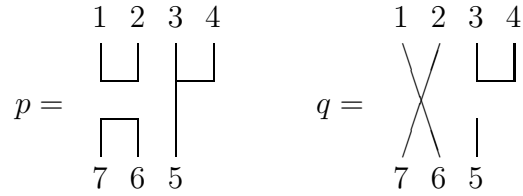
\setlength{\unitlength}{0.5cm}
\begin{center}
\begin{picture}(12,5)
\put(-2, 2){$p=$}
\put(-1,4.85){\partii{1}{1}{2}}
\put(-1,4.85){\partii{1}{3}{4}}
\put(-1,0.85){\uppartii{1}{1}{2}}
\put(-1,0.85){\upparti{2}{3}}
\put(0.05,0){7}
\put(1.05,0){6}
\put(2.05,0){5}
\put(0.05,4.3){1}
\put(1.05,4.3){2}
\put(2.05,4.3){3}
\put(3.05,4.3){4}
\put(6,2){$q=$}
\put(8.3,3.85){\line(1,-3){1}}
\put(7,4.85){\partii{1}{3}{4}}
\put(8.3,0.85){\line(1,3){1}}
\put(7,0.85){\upparti{1}{3}}
\put(8.05,0){7}
\put(9.05,0){6}
\put(10.05,0){5}
\put(8.05,4.3){1}
\put(9.05,4.3){2}
\put(10.05,4.3){3}
\put(11.05,4.3){4}
\end{picture}
\end{center}
\end{ex}

We usually omit to write the numbers in the picture. If the strings of a partition may be drawn in such a way that they do not cross, we call it a \emph{noncrossing partition}, denoting by $NC\subset P$ the subset of all noncrossing partitions. Note that in Example \ref{Expq}, the partition $p$ is in $NC$ while $q$ is not.

\begin{ex}
Here are some examples of partitions in $P$.
\begin{itemize}
\item[(a)] The \emph{identity partition} $\idpart\in P(1,1)$.
\item[(b)] The \emph{pair partitions} $\paarpart\in P(0,2)$ and $\baarpart\in P(2,0)$.
\item[(c)] The \emph{singleton partitions} $\singleton\in P(0,1)$ and $\downsingleton\in P(1,0)$.
\end{itemize}
\end{ex}

Partitions are well-known objects in mathematics, see for instance \cite{St, NS, BS, TWcomb}.

\subsection{Spatial partitions}
\label{SectAdapted}

Let us now introduce the new notion of spatial partitions. Let $m\in\N$ and $k,l\in\N_0$. Consider the set
\[\{1,\ldots,k,k+1,\ldots,k+l\}\times\{1,\ldots,m\}.\]
A \emph{spatial partition (on $m$ levels)} is a decomposition of this set into disjoint subsets (\emph{blocks}). We sometimes also simply write \emph{partition}, when it is clear that we speak of spatial partitions.
The set of all such spatial partitions is denoted by $\Pm(k,l)$ and we put $\Pm:=\bigcup_{k,l\in\N_0}\Pm(k,l)$.
Again, the points $(1,y),\ldots,(k,y)$ are seen as upper points and the points $(k+1,y),\ldots,(k+l,y)$ as lower ones, for $y\in\{1,\ldots,m\}$. Furthermore, if $(x,y)\in \{1,\ldots,k,k+1,\ldots,k+l\}\times\{1,\ldots,m\}$ is a point of a partition $p\in\Pm(k,l)$, we call its second component $y$ the \emph{level} of the point. We say that a partition $p\in\Pm$ \emph{respects the levels}, if whenever two points $(x_1,y_1)$ and $(x_2,y_2)$ are in the same block of $p$, then $y_1=y_2$.

We view a spatial partition as a three-dimensional partition having an upper plane consisting of $k\times m$ points and a lower plane of $l\times m$ points. Thus,  the $m$ levels are nothing but a new dimension in our pictorial representation.

\begin{ex}
Let $m=3$, $k=2$ and $l=4$.
 The following partitions are in $P^{(3)}(2,4)$ and $p$ respects the levels while $q$ does not.
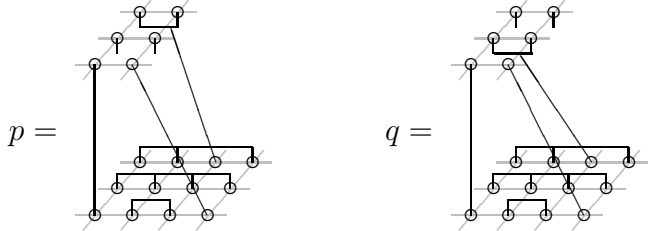
\setlength{\unitlength}{0.5cm}
\begin{center}
\begin{picture}(20,8)
\put(0,3){$p=$}
\newsavebox{\BoxP}
   \savebox{\BoxP}
   { \begin{picture}(10,8)
\put(2,1){\usebox{\boxGrLLL}}
\put(3,1){\usebox{\boxGrLLL}}
\put(4,1){\usebox{\boxGrLLL}}
\put(5,1){\usebox{\boxGrLLL}}
\put(2,5){\usebox{\boxGrLLL}}
\put(3,5){\usebox{\boxGrLLL}}
\put(2,1){\line(0,1){4}}
\put(3,1){\line(0,1){0.4}}
\put(4,1){\line(0,1){0.4}}
\put(3,1.4){\line(1,0){1}}
\put(5,1){\line(-1,2){2}}
\put(2.6,1.7){\line(0,1){0.4}}
\put(3.6,1.7){\line(0,1){0.4}}
\put(4.6,1.7){\line(0,1){0.4}}
\put(5.6,1.7){\line(0,1){0.4}}
\put(2.6,2.1){\line(1,0){3}}
\put(3.2,2.4){\line(0,1){0.4}}
\put(4.2,2.4){\line(0,1){0.4}}
\put(6.2,2.4){\line(0,1){0.4}}
\put(3.2,2.8){\line(1,0){3}}
\put(2.6,5.7){\line(0,-1){0.4}}
\put(3.6,5.7){\line(0,-1){0.4}}
\put(3.2,6.4){\line(0,-1){0.4}}
\put(4.2,6.4){\line(0,-1){0.4}}
     \end{picture}}      
\put(0,0){\usebox{\BoxP}}  
\put(5.5,2.4){\line(-1,3){1.2}}   
\put(3.5,6){\line(1,0){1}}
\put(10,3){$q=$}
\put(10,0){\usebox{\BoxP}}     
\put(15.5,2.4){\line(-2,3){1.9}}   
\put(12.9,5.3){\line(1,0){1}}
\end{picture}
\end{center}
\end{ex}

We observe that there is no canonical definition of noncrossing partitions in three dimensions.

\begin{rem}\label{RemPmAndP}
For any $m\in\N$, $k,l\in\N_0$, the sets 
\[A:=\{1,\ldots,km,km+1,\ldots,km+lm\}\]
and
\[B:=\{1,\ldots,k,k+1,\ldots,k+l\}\times\{1,\ldots,m\}\]
are in bijective correspondence by identifying a point $(x-1)m+y\in A$, $1\leq x\leq k+l$, $1\leq y\leq m$ with the point $(x,y)\in B$. Thus, the sets $\Pm(k,l)$ and  $P(km,lm)$ are isomorphic. In particular, for $m=1$, spatial partitions (on one level) are simply the well-known partitions in the sense of Section \ref{SectPartitions}.
\end{rem}

\begin{defn}\label{DefmAmplifiedPart}
If $p\in P(k,l)$ is a partition, then $p^{(m)}\in \Pm(k,l)$ given by repeating $p$ on each level $1\leq s\leq m$ is the \emph{amplified version of $p$ (on $m$ levels)}. It respects the levels.
\end{defn}

\begin{ex}\label{ExIdPaar}
The amplified partitions $\idpart^{(4)}$ and $\paarpart^{(4)}$ are the following partitions.\setlength{\unitlength}{0.5cm}
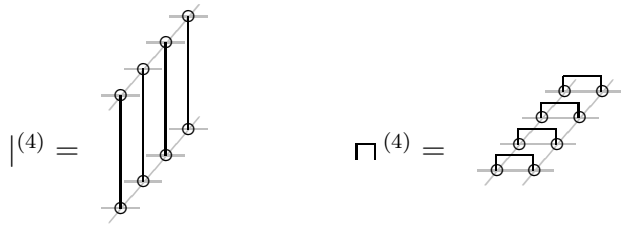
\begin{center}
\begin{picture}(18,8)
\put(0,2.5){$\idpart^{(4)}=$}
\put(3,1){\usebox{\boxGrLLLL}}
\put(3,4){\usebox{\boxGrLLLL}}
\put(3,1.2){\line(0,1){3}}
\put(3.6,1.9){\line(0,1){3}}
\put(4.2,2.6){\line(0,1){3}}
\put(4.8,3.3){\line(0,1){3}}
\put(9,2.5){$\paarpart^{(4)}=$}
\put(13,2){\usebox{\boxGrLLLL}}
\put(14,2){\usebox{\boxGrLLLL}}
   \savebox{\BoxPa}
   { \begin{picture}(1,1)
   \put(0,0){\line(0,1){0.4}}
   \put(1,0){\line(0,1){0.4}}
   \put(0,0.4){\line(1,0){1}}
   \end{picture}}
\put(12.7,2.2){\usebox{\BoxPa}}
\put(13.3,2.9){\usebox{\BoxPa}}
\put(13.9,3.6){\usebox{\BoxPa}}
\put(14.5,4.3){\usebox{\BoxPa}}
\end{picture}
\end{center}
\end{ex}

\subsection{Categories of spatial partitions}
\label{SectMCateg}

For a fixed $m\in\N$, we have the following operations on the set $\Pm$, the so called \emph{category operations}.
\begin{itemize}
 \item  The \emph{tensor product} of two spatial partitions $p\in \Pm(k,l)$ and $q\in \Pm(k',l')$ is the spatial partition $p\otimes q\in \Pm(k+k',l+l')$ obtained by writing $p$ and $q$ side by side.
 \item The \emph{composition} of two spatial partitions $q\in \Pm(k,r)$ and $p\in \Pm(r,l)$ is the spatial partition  $pq\in \Pm(k,l)$ obtained by writing $p$ below $q$, joining their strings by identifying the lower resp. upper  $r\times m$-planes of points, and erasing the strings which are disconnected from the upper $k\times m$-plane and the lower $l\times m$-plane.
 \item The \emph{involution} of a spatial partition $p\in \Pm(k,l)$ is given by the spatial partition $p^*\in \Pm(l,k)$ obtained when swapping the upper with the lower plane.
 \end{itemize}
 
\begin{defn}\label{DefmCateg}
A subset $\CC\subset \Pm$ is a \emph{category of spatial partitions}, if $\CC$ is closed under tensor product, composition and involution, and if it contains the amplified identity partition $\idpart^{(m)}\in \Pm(1,1)$ and the amplified pair partition $\paarpart^{(m)}\in \Pm(0,2)$. 
\end{defn}

 We write $\CC=\langle p_1,\ldots, p_n\rangle$, if $\CC$ is the smallest category containing $p_1,\ldots,p_n\in \Pm$. We then speak of the category \emph{generated by} $p_1,\ldots,p_n$. We omit to write $\paarpart^{(m)}$ and $\idpart^{(m)}$ as generators since they are always contained in a category.

\begin{rem}\label{RemBaseCase}
In the case $m=1$, the above category operations as well as categories of partitions were first introduced by Banica and Speicher \cite{BS}; see also  \cite{TWcomb,VSW} for concrete examples of these operations in that case. We now extend their definition to the three-dimensional setting in a canonical way, but let us note another aspect of the passage from $m=1$ to arbitrary $m\in\N$. Observe that the isomorphism $\Pm(k,l)\cong P(km,lm)$ of  Remark \ref{RemPmAndP} respects the category operations. Hence, if we view $\Pm$ as a subset of $P$, a category $\CC\subset \Pm$ of spatial partitions corresponds to a set $\CC'\subset P$ which is closed under the category operations (as operations in $P$). However, $\CC'$ is \emph{not} a category of partitions in Banica-Speicher's sense, since it does not contain the base partitions $\paarpart$ nor $\idpart$. From this point of view, we somehow modified Banica and Speicher's definition of  categories of partitions $\CC\subset P$ by simply replacing the base partitions $\paarpart\in P$ and $\idpart\in P$ by different ones, namely by
\[\paarpart^{(m)}\in \Pm(0,2) \quad\longleftrightarrow \quad\{\{1,m+1\},\{2,m+2\},\ldots,\{m,2m\}\}\in P(0,2m)\]
and
\[\idpart^{(m)}\in \Pm(1,1) \quad\longleftrightarrow \quad \idpart^{\otimes m}\in P(m,m)\]
using the isomorphism $\Pm(k,l)\cong P(km,lm)$ of Remark \ref{RemPmAndP}.
From the combinatorial point of view, there is no difficulty in choosing different base partitions for Banica and Speicher's categories of partitions, but so far a quantum group interpretation of such categories was missing. In this article, we provide one for the case of $\paarpart^{(m)}$ and $\idpart^{(m)}$.
\end{rem}

\begin{defn}
Let $\CC\subset P$ be a set of partitions. Using the notation of Definition \ref{DefmAmplifiedPart}, we denote by 
\[[\CC]^{(m)}:=\{p^{(m)}\;|\; p\in\CC\}\subset \Pm\]
 the \emph{amplification} of $\CC$.
\end{defn}
 
\begin{lem}\label{LemAmplif}
If $\CC\subset P$ is a category of partitions, then the amplification $[\CC]^{(m)}\subset \Pm$ is a category of spatial partitions.
\end{lem}
\begin{proof}
A direct proof is straightforward. Alternatively, one can use the fact that the isomorphism of Remark \ref{RemPmAndP} respects the category operations.
\end{proof}
 
\begin{ex}\label{ExCateg}
Here are examples of categories of spatial partitions. We will see more exotic ones in Section \ref{SectExamples}.
\begin{itemize}
\item[(a)] The set $\Pm$ of all spatial partitions is a category of spatial partitions. It is maximal in the sense that it contains all other categories of spatial partitions.
We have $[P]^{(m)}\neq \Pm$ for $m\neq 1$, since a spatial partition in the amplification $[P]^{(m)}$ of $P$ consists of $m$ copies of a partition from $P$ to all levels; the set $\Pm$ in turn is much larger containing \emph{any} spatial partition.
\item[(b)] The set $\Pm_2$ of all spatial pair partitions (i.e. all blocks consist of exactly two points) is a category of spatial partitions. Again, we have $[P_2]^{(m)}\neq \Pm_2$ for $m\neq 1$.
\item[(c)] The amplification $[NC_2]^{(m)}$ of $NC_2$ is the minimal category of spatial partitions. It is generated by $\idpart^{(m)}$ and $\paarpart^{(m)}$.
Note that $[NC_2]^{(m)}$ is \emph{not} the set of all noncrossing pair partitions in $P^{(m)}$. In fact, it is not clear in the three-dimensional picture what a noncrossing partition is supposed to be -- only the identification $P^{(m)}(k,l)\cong P(km,lm)$ allows for a notion of noncrossing partitions. However, the set $\bigcup_{k,l\in\N_0} NC(km,lm)$ seen as a subset of $P^{(m)}$ is not a category of spatial partitions. It is closed under the category operations, but it does not contain $\paarpart^{(m)}$ (see Remark \ref{RemBaseCase}).
\end{itemize}
\end{ex}

Let us mention  another useful operation on the set $\Pm$. 
We define the \emph{$m$-rotation} by the following. Let $p\in \Pm(k,l)$ and consider the upper plane of points of $p$ consisting of $k$ rows, each row consisting of $m$ points.
Let $q\in \Pm(k-1,l+1)$ be the spatial partition obtained from $p$ by shifting the leftmost upper row of $p$ to the left of the lower plane without changing the order of the points nor the strings attached to these points. We say that $q$ is a rotated version of $p$. Likewise we may rotate on the right hand side and we may rotate lower points to the upper plane. As an example, observe that $\paarpart^{(m)}$ is obtained from $\idpart^{(m)}$ by $m$-rotation. We refer to \cite[Sect. 1.2]{TWcomb} for examples of $1$-rotation. Note that the $m$-rotation does not affect the level of a point when being rotated.

\begin{lem}
Every category of spatial partitions is closed under $m$-rotation.
\end{lem}
\begin{proof}
Let $p\in \Pm(k,l)$. The spatial partition $q:=(\idpart^{(m)}\otimes p)(\paarpart^{(m)}\otimes\idpart^{((k-1)m)})$ arises from $p$ by $m$-rotating the leftmost upper row of $p$ to the lower plane. Similarly for the rotations from the lower plane to the upper plane, and for rotations on the right hand side. See also \cite[Lem. 1.1]{TWcomb}
\end{proof}

\subsection{$\pi$-graded spatial partitions and categories}\label{SectPiGraded}

Later, we will need spatial partitions that are allowed to mix certain levels -- but not all levels. This is captured by the following definition. The idea is to cluster into a block each set of levels that are allowed to be interchanged  and to decompose the set $\{1,\ldots,m\}$ accordingly. This is encoded in a partition $\pi\in P(m)$, the grading partition.

\begin{defn}\label{DefPiGraded}
Let $m\in \N$ and let $\pi\in P(m)$ be a partition of $m$ points. A spatial partition $p\in \Pm$ is \emph{$\pi$-graded}, if whenever two points $(x_1,y_1)$ and $(x_2,y_2)$ are in the same block of $p$, then $y_1$ and $y_2$ are in the same block of $\pi$. The partition $\pi$ is called the \emph{grading partition}. We denote by $\Pm_\pi$ the set of all $\pi$-graded (spatial) partitions in $\Pm$.
\end{defn}

If $\pi$ consists only of singletons, the $\pi$-graded partitions are exactly those that respect the levels. If $\pi$ is the one block partition, then every partition in $\Pm$ is $\pi$-graded. As a nontrivial example, let $m=4$ and $\pi=\{\{1,3\},\{2,4\}\}$. Then a partition $p\in\Pm$ is $\pi$-graded if and only if no block of $p$ contains points from an odd level \emph{and} from an even level.

\begin{defn}
A category of spatial partitions $\CC$ is \emph{$\pi$-graded}, if all partitions in $\CC$ are $\pi$-graded. 
\end{defn}

\begin{lem}\label{LemPiGrading}
Let $m\in\N$.
\begin{itemize}
\item[(a)] Let $\pi\in P(m)$ be a grading partition. If $p$ and $q$ in $\Pm$ are $\pi$-graded, then so are $p\otimes q, pq, p^*$ or any $m$-rotation of $p$ or $q$. 
\item[(b)] The set $\Pm_\pi$ of all $\pi$-graded partitions in $\Pm$ is a category of spatial partitions.
\item[(c)] If $p_1,\ldots,p_k$ are $\pi$-graded, so is the category $\langle p_1,\ldots,p_k\rangle$ generated by them.
\end{itemize}
\end{lem}
\begin{proof}
The proof of (a) is straightforward, and (b) and (c) follow immediately. 
\end{proof}

\subsection{Generators of $\Pm_{\pi}$, $P^{(2)}$ and $P_2^{(2)}$}

In the case $m=1$, it is not difficult to see that $P$ is generated by  $\vierpartrot$, $\singleton$ and $\crosspart$. This allows us to define natural further categories like $\langle\vierpartrot\rangle$ or $\langle\vierpartrot, \singleton\rangle$, see for instance \cite{Web}. We are thus interested in finding canonical generators of the category $\Pm$, the maximal category of spatial partitions. We refine the statement by considering $\pi$-graded partitions, including the case $\Pm$ when $\pi$ is the one block partition on $m$ points.

\begin{thm}\label{PropGenerators}
Let $\pi\in P(m)$ be a grading partition. The category $\Pm_\pi$ of all $\pi$-graded partitions
is generated by the following partitions besides the base partitions $\idpart^{(m)}$ and $\paarpart^{(m)}$:
\begin{itemize}
\item[(i)] The singleton partition $\singleton^{(m)}$.
\item[(ii)] For $i=1,\ldots,m$, the partition given by $\vierpartrot$ on level $i$ and $\idpart\otimes\idpart$ on all other levels. For $m=2$ this amounts to $\FourOneIdTwopart\in P^{(2)}(2,2)$ and $\IdOneFourTwopart\in P^{(2)}(2,2)$.
\item[(iii)]  For $i=1,\ldots,m$, the partition given by $\crosspart$ on level $i$ and $\idpart\otimes\idpart$ on all other levels. For $m=2$ this amounts to $\CrossOneIdTwopart\in P^{(2)}(2,2)$ and $\IdOneCrossTwopart\in P^{(2)}(2,2)$.
\item[(iv)] For $1\leq i<j\leq m$ two points in the same block of $\pi$, the partition given by $\FourOneTwopart$ on the levels $i$ and $j$ and $\idpart$ on the others. For $m=2$ and $\pi=\paarpart$ this amounts to $\FourOneTwopart\in P^{(2)}(1,1)$.
\end{itemize}
\end{thm}
\begin{proof}
We give a proof for $m=2$ and $\pi=\paarpart$, the general case being a straightforward adaption.

Let $\CC\subset P^{(2)}$ be the category generated by (i) to (iv). Let $p_1$ and $q_2$ be partitions in $P(k,l)$. Using $\FourOneIdTwopart$,  $\SingOneSingTwopart$, $\CrossOneIdTwopart$, the base partitions, and the category operations, we may construct a partition $p\in \CC$ respecting the levels, such that on level one, we have $p_1$ (since $P=\langle\vierpartrot,\singleton,\crosspart\rangle$). Likewise, we produce a partition $q\in \CC$ respecting the levels, such that on  level two, we have $q_2$. Using (iii), we may permute the points of $p\otimes q\in \CC$ in order to obtain a partition $r\otimes s\in\CC$ respecting the levels with $r,s\in P^{(2)}(k,l)$ such that $r$ restricts to $p_1$ on  level one and to $q_2$ on  level two. Composing this partition with $\singleton^{(2)}$ and its adjoint, we infer $r\in\CC$.

Use $\FourOneTwopart$ and (iii) to connect arbitrary upper points of $p_1$ with arbitrary upper points of $q_2$, and likewise for connecting lower points with lower points. As for building a string between an upper point of $p_1$ and a lower point of $q_2$, assume that both are leftmost within their  level (possibly using (iii)). Let $v\in P^{(2)}(1,2)$ be the partition consisting of a three block on level one and $\singleton\otimes \idpart$ on  level two. Let $w\in P^{(2)}(2,1)$ be the partition consisting of $\downsingleton\otimes\idpart$ on  level one and a three block on  level two. By the preceding considerations, $v$ and $w$ are in $\CC$. We conclude that the partition $r':=(w\otimes (\idpart^{(2)})^{\otimes l-1})(\FourOneTwopart\otimes r)(v\otimes (\idpart^{(2)})^{\otimes k-1})$ is in $\CC$. 
\setlength{\unitlength}{0.5cm}
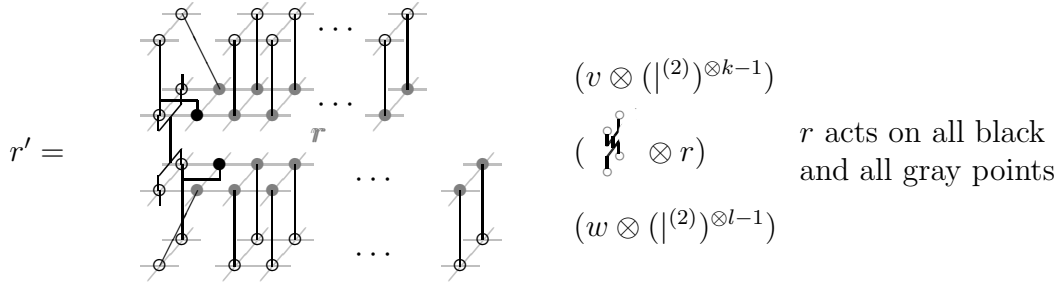
\begin{center}
\begin{picture}(28,8)
\newsavebox{\boxGrEinsa}
   \savebox{\boxGrEinsa}
   { \begin{picture}(6,2)
   \put(-0.3,0){\usebox{\boxGrLL}}
   \put(1.7,0){\usebox{\boxGrLL}}
   \put(2.7,0){\usebox{\boxGrLL}}
   \put(3.9,0.3){$\cdots$}
   \put(5.7,0){\usebox{\boxGrLL}}
      \end{picture}}
\newsavebox{\boxGrEinsb}
   \savebox{\boxGrEinsb}
   { \begin{picture}(6,2)
   \put(-0.3,0){\usebox{\boxGrLL}}
   \put(0.7,0){\usebox{\boxGrLLsg}}
   \put(1.7,0){\usebox{\boxGrLLgg}}
   \put(2.7,0){\usebox{\boxGrLLgg}}
   \put(3.9,0.3){$\cdots$}
   \put(5.7,0){\usebox{\boxGrLLgg}}
      \end{picture}}
\newsavebox{\boxGrZweia}
   \savebox{\boxGrZweia}
   { \begin{picture}(8,2)
   \put(-0.3,0){\usebox{\boxGrLL}}
   \put(1.7,0){\usebox{\boxGrLL}}
   \put(2.7,0){\usebox{\boxGrLL}}
   \put(4.9,0.3){$\cdots$}
   \put(7.7,0){\usebox{\boxGrLL}}
      \end{picture}}       
\newsavebox{\boxGrZweib}
   \savebox{\boxGrZweib}
   { \begin{picture}(8,2)
   \put(-0.3,0){\usebox{\boxGrLL}}
   \put(0.7,0){\usebox{\boxGrLLgs}}
   \put(1.7,0){\usebox{\boxGrLLgg}}
   \put(2.7,0){\usebox{\boxGrLLgg}}
   \put(4.9,0.3){$\cdots$}
   \put(7.7,0){\usebox{\boxGrLLgg}}
      \end{picture}}      
\newsavebox{\boxIdid}
   \savebox{\boxIdid}
   { \begin{picture}(1,3)
   \put(-0.3,0.2){\line(0,1){2}}
   \put(0.3,0.9){\line(0,1){2}}
      \end{picture}}      
\newsavebox{\boxS}
   \savebox{\boxS}
   { \begin{picture}(2,2)
   \put(-0.3,0.2){\line(0,1){0.4}}
   \put(0.3,0.9){\line(0,1){0.4}}
      \end{picture}}        
\newsavebox{\boxVPs}
   \savebox{\boxVPs}
   { \begin{picture}(2,2)
   \put(-0.3,0){\usebox{\boxS}}
   \put(-0.3,1.6){\usebox{\boxS}}
   \put(-0.3,0.6){\line(5,6){0.6}}
   \put(-0.3,1.8){\line(5,6){0.6}}
   \put(0,0.95){\line(0,1){1.2}}
      \end{picture}}     
\put(0,3){$r'=$}
\put(4,6){\usebox{\boxGrEinsa}}
\put(4,4){\usebox{\boxGrEinsb}}
\put(6,4){\usebox{\boxIdid}}
\put(7,4){\usebox{\boxIdid}}
\put(10,4){\usebox{\boxIdid}}
\put(4,4.2){\line(0,1){2}}
\put(5,4.2){\line(0,1){0.4}}
\put(4,4.6){\line(1,0){1}}
\put(4.6,4.9){\line(0,1){0.4}}
\put(5.6,4.9){\line(-1,2){1}}
\put(4,2){\usebox{\boxGrZweib}}
\put(4,2){\usebox{\boxVPs}}
\put(8,3.5){{\color{gray}$r$}}
\put(8.05,3.5){{\color{gray}$r$}}
\put(4,0){\usebox{\boxGrZweia}}
\put(6,0){\usebox{\boxIdid}}
\put(7,0){\usebox{\boxIdid}}
\put(12,0){\usebox{\boxIdid}}
\put(4,0.2){\line(1,2){1}}
\put(4,2.2){\line(0,-1){0.4}}
\put(4.6,0.9){\line(0,1){2}}
\put(5.6,2.9){\line(0,-1){0.4}}
\put(4.6,2.5){\line(1,0){1}}
\put(15,1){$(w\otimes (\idpart^{(2)})^{\otimes l-1})$}
\put(15,3){$(\FourOneTwopart\otimes r)$}
\put(15,5){$(v\otimes (\idpart^{(2)})^{\otimes k-1})$}
\put(21,3.5){$r$ acts on all black}
\put(21,2.5){and all gray points}
\end{picture}
\end{center}
It coincides with the partition obtained when connecting the leftmost upper point of $r$ on level one with the leftmost lower point of $r$ on level two (both points are marked in black in the above picture).
We infer that we may connect arbitrary blocks of $p_1$ with arbitrary blocks of $q_2$, such that we may construct any partition $p\in P^{(2)}$ in $\CC$.
\end{proof}

\begin{cor}\label{CorGenerators}
For $m=2$ and $\pi=\paarpart$, the category $P^{(2)}$ is generated by the following partitions besides the base partitions $\idpart^{(2)}$ and $\paarpart^{(2)}$:
\begin{itemize}
\item[(i)] $\singleton^{(2)}\in P^{(2)}(0,1)$,
\item[(ii)] $\FourOneIdTwopart\in P^{(2)}(2,2)$ and $\IdOneFourTwopart\in P^{(2)}(2,2)$,
\item[(iii)]  $\CrossOneIdTwopart\in P^{(2)}(2,2)$ and $\IdOneCrossTwopart\in P^{(2)}(2,2)$,
\item[(iv)] $\FourOneTwopart\in P^{(2)}(1,1)$.
\end{itemize}
Moreover, we can replace the partition of item (iv) by 
\begin{itemize}
\item[(iv')] $\PaarOneTwopart\in P^{(2)}(0,1)$.
\end{itemize}
\end{cor}
\begin{proof}
By Theorem \ref{PropGenerators}, the category $P^{(2)}$ is generated by (i-iv). Thus, all we have to prove is that $\FourOneTwopart$ is in the category generated by (i-iii) and (iv'), which is the case as can be seen by the following picture.
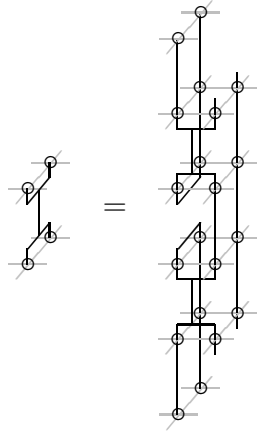
\setlength{\unitlength}{0.5cm}
\begin{center}
\begin{picture}(8,12)
\newsavebox{\boxVierer}
   \savebox{\boxVierer}
   { \begin{picture}(2,2)
   \put(-0.3,0){\usebox{\boxGrLL}}
   \put(0.7,0){\usebox{\boxGrLL}}
      \end{picture}}
   \savebox{\boxIdid}
   { \begin{picture}(1,3)
   \put(-0.3,0.2){\line(0,1){2}}
   \put(0.3,0.9){\line(0,1){2}}
      \end{picture}}      
\newsavebox{\boxIdidd}
   \savebox{\boxIdidd}
   { \begin{picture}(1,3)
   \put(0.3,0.9){\line(0,1){2}}
   \put(1.3,0.9){\line(0,1){2}}   
      \end{picture}}      
\newsavebox{\boxVP}
   \savebox{\boxVP}
   { \begin{picture}(2,2)
   \put(-0.3,0.2){\line(0,1){0.4}}
   \put(0.7,0.2){\line(0,1){0.4}}
   \put(-0.3,0.6){\line(1,0){1}}   
   \put(0.1,0.6){\line(0,1){1.2}}
   \put(-0.3,1.8){\line(0,1){0.4}}
   \put(0.7,1.8){\line(0,1){0.4}}
   \put(-0.3,1.8){\line(1,0){1}}   
      \end{picture}}     
   \savebox{\boxS}
   { \begin{picture}(2,2)
\put(-0.3,0.2){\line(0,1){0.4}}
\put(0.3,0.9){\line(0,1){0.4}}
      \end{picture}}     
\put(4,0){\usebox{\boxGrLL}}
\put(4,0){\usebox{\boxIdid}}
\put(5,1.6){\usebox{\boxS}}
\put(4,2){\usebox{\boxVierer}}
\put(4,2){\usebox{\boxIdidd}}
\put(4,2){\usebox{\boxVP}}
\put(4,4){\usebox{\boxVierer}}
\put(4,4){\usebox{\boxS}}
\put(4,5.6){\usebox{\boxS}}
\put(4,4.6){\line(5,6){0.6}}
\put(4,5.8){\line(5,6){0.6}}
\put(5,4){\usebox{\boxIdid}}
\put(4,6){\usebox{\boxVierer}}
\put(4,6){\usebox{\boxVP}}
\put(4,6){\usebox{\boxIdidd}}
\put(4,8){\usebox{\boxVierer}}
\put(4,8){\usebox{\boxIdid}}
\put(5,8){\usebox{\boxS}}
\put(4,10){\usebox{\boxGrLL}}
\put(0,4){\usebox{\boxGrLL}}
\put(0,6){\usebox{\boxGrLL}}
\put(0,4){\usebox{\boxS}}
\put(0,5.6){\usebox{\boxS}}
\put(0,4.6){\line(5,6){0.6}}
\put(0,5.8){\line(5,6){0.6}}
\put(0.3,4.95){\line(0,1){1.2}}
\put(2,5.5){$=$}
\end{picture}
\end{center}
\end{proof}

For $m=1$, we have $P_2=\langle\crosspart\rangle$, see \cite{VSW}. For $m=2$, the situation is more complicated.

\begin{thm}\label{PropPZwei}
For $m=2$ and $\pi=\paarpart$, the category $P_2^{(2)}$ consisting of all spatial pair partitions on two levels (see also Example \ref{ExCateg}(b)) is generated by the following partitions besides the base partitions $\idpart^{(2)}$ and $\paarpart^{(2)}$:
\begin{itemize}
\item[(i)] $\PaarPaarOneIdTwopart\in P^{(2)}(2,2)$ and $\IdOnePaarPaarTwopart\in P^{(2)}(2,2)$,
\item[(ii)]  $\CrossOneIdTwopart\in P^{(2)}(2,2)$ and $\IdOneCrossTwopart\in P^{(2)}(2,2)$,
\item[(iii)] $\PaarOneTwopart\in P^{(2)}(0,1)$.
\end{itemize}
\end{thm}
\begin{proof}
Similar to the proof of Theorem \ref{PropGenerators}, we use $\PaarPaarOneIdTwopart$, $\CrossOneIdTwopart$ and the base partitions in order to construct arbitrary pair partitions $pp^*\in P_2(k,k)$ (where $p\in P_2(0,k)$) on level one. Tensoring several such partitions and several copies of $\idpart^{(2)}$, using (ii) to permute the points and finally composing them with suitable tensor powers of $\paarpart^{(2)}$ and its adjoint, we obtain any arbitrary partition in $P_2^{(2)}$ respecting the levels. We may mix the levels using the partition $\CrossOneTwopart$ which may be constructed from (ii) and (iii).
\end{proof}

\section{Spatial partition quantum groups}

We will now associate quantum groups to categories of spatial partitions. We first recall some basics about compact matrix quantum groups and Woronowicz's Tannaka-Krein result.

\subsection{Compact matrix quantum groups}

The following definition of a compact matrix quantum group is due to Woronowicz \cite{WoCMQG, WoRemark}. It is a special case of his theory of compact quantum groups. See also \cite{Tim, Nesh} for more details.

\begin{defn}\label{DefCMQG}
Let $n\in\N$. 
A \emph{compact matrix quantum group} is a tupel  $(A,u)$ such that
\begin{itemize}
\item $A$ is a unital $C^*$-algebra generated by $n^2$ elements $u_{ij}$, $1\leq i,j\leq n$,
\item the matrices $u=(u_{ij})$ and $\bar u=(u_{ij}^*)$ are invertible in $M_{n}(A)$,
\item and the map $\Delta:A\to A\otimes_{\min}A$ given by $\Delta(u_{ij})=\sum_k u_{ik}\otimes u_{kj}$ is a $^*$-homomorphism.
\end{itemize}
\end{defn}

If $G\subset M_n(\C)$ is a compact matrix group, then $C(G)$ gives rise to a compact matrix quantum group in the above sense. We therefore write $A=C(G)$ even if the $C^*$-algebra $A$ from Definition \ref{DefCMQG} is noncommutative, and we speak of $G$ as the compact matrix quantum group (which is only defined via $C(G)$), sometimes specifying $(G,u)$ in order to keep track of the generating matrix $u$. 

\begin{defn}\label{DefSubgroup}
Let $(G,u)$ with $u=(u_{ij})_{i,j=1,\ldots,n}$ and $(H,v)$ with $v=(v_{ij})_{i,j=1,\ldots,m}$ be two compact matrix quantum groups.
\begin{itemize}
\item[(a)] We say that $G$ is a \emph{quantum subgroup of $H$ as a compact matrix quantum group}, if there is a surjective $^*$-homomorphism $\phi:C(H)\to C(G)$ mapping $\phi(v_{ij})=u_{ij}$. In particular, we require $n=m$.
\item[(b)] We say that $G$ is a \emph{quantum subgroup of $H$ as a compact quantum group}, if there is a surjective $^*$-homomorphism $\phi:C(H)\to C(G)$ such that $\Delta_G(\phi(v_{ij}))=\sum_{k=1}^m\phi(v_{ik})\otimes\phi(v_{kj})$. In general, we may have $n\neq m$.
\end{itemize}
We  simply speak of a \emph{quantum subgroup}  $G\subset H$ if there is no confusion with the above cases (a) and (b); in particular, since we will always apply throughout the article (a) in the case $n=m$ and (b) in the case $n\neq m$.
\end{defn}

\begin{ex}\label{ExCMQG}
\begin{itemize}
\item[(a)] Wang \cite{WangOrth} defined the \emph{free orthogonal quantum group} $O_n^+$ as the compact matrix quantum group given by the universal $C^*$-algebra
\[C(O_n^+):=C^*(u_{ij}, i,j=1,\ldots,n\;|\; u_{ij}=u_{ij}^*, \sum_k u_{ik}u_{jk}=\sum_k u_{ki}u_{kj}=\delta_{ij}).\]
If we take the quotient of $C(O_n^+)$ by the relations that all $u_{ij}$ commute, we obtain the algebra of functions $C(O_n)$ over the orthogonal group $O_n\subset M_n(\C)$. Thus, we have $O_n\subset O_n^+$ as quantum subgroups in the sense of Definition \ref{DefSubgroup}(a).
\item[(b)] Wang \cite{WangSym} also defined the \emph{free symmetric quantum group} $S_n^+$ via
\[C(S_n^+):=C^*(u_{ij}, i,j=1,\ldots,n\;|\; u_{ij}=u_{ij}^*=u_{ij}^2, \sum_k u_{ik}=\sum_k u_{kj}=1).\]
The quotient by the commutator ideal yields $C(S_n)$, where $S_n\subset M_n(\C)$ is the symmetric group, thus $S_n\subset S_n^+$. It is not difficult to check, that we have $u_{ik}u_{jk}=0$ and $u_{ki}u_{kj}=0$ for $i\neq j$. Hence $S_n^+$ is a quantum subgroup of $O_n^+$ in the sense of Definition \ref{DefSubgroup}(a).
\item[(c)] We may view the symmetric group $S_n$ as a quantum subgroup of $O_{n^2}^+$ in the sense of Definition \ref{DefSubgroup}(b). Indeed, let $u_{ij}$, $i,j=1,\ldots,n$ be the generators of $C(S_n)$. For $i_1,i_2,j_1,j_2\in\{1,\ldots,n\}$, we put
\[v'_{(i_1,i_2)(j_1,j_2)}:=u_{i_1j_1}u_{i_2j_2}\in C(S_n).\]
Labeling the generators of $C(O_{n^2}^+)$ by $v_{(i_1,i_2)(j_1,j_2)}$, it is easy to verify that we have a surjection from $C(O_{n^2}^+)$ to $C(S_n)$ mapping $v_{(i_1,i_2)(j_1,j_2)}$ to $v'_{(i_1,i_2)(j_1,j_2)}$. It respects the comultiplication map $\Delta$ of $S_n$.
\end{itemize}
\end{ex}

\subsection{Tannaka-Krein duality}

Similar to Schur-Weyl duality for groups, we may reconstruct a compact matrix quantum group from its intertwiner spaces. This is due to Woronowicz's Tannaka-Krein result \cite{WoTK}. Let us briefly sketch it, referring to \cite{TW, Nesh, Mal} for more details. We restrict to the case $u=\bar u$.

For $k\in\N_0$, the matrix
\[u^{\otimes k}\in M_{n^k}(C(G))\cong M_n(\C)\otimes\cdots \otimes M_n(\C)\otimes C(G)\]
is a \emph{(tensor) representation} of a compact matrix quantum group $(G,u)$.  The set of \emph{intertwiners} between $u^{\otimes k}$ and $u^{\otimes l}$ is the set of linear map $T:(\C^n)^{\otimes k}\to (\C^n)^{\otimes l}$ such that $Tu^{\otimes k}=u^{\otimes l}T$. It is denoted by $\Hom_G(k,l)$. The collection of spaces $(\Hom_G(k,l))_{k,l\in\N_0}$ forms a \emph{$C^*$-tensor category} or rather a \emph{concrete monoidal $W^*$-category} in the sense of Woronowicz. A simplifed version of his Tannaka-Krein Theorem is the following.

\begin{thm}[{Tannaka-Krein Theorem \cite{WoTK}}]
Let $(\Hom(k,l))_{k,l\in\N_0}$ be an (abstract) $C^*$-tensor category which is generated by an element $f=\bar f$. Then, there exists a compact matrix quantum group $(G,u)$ with $u=\bar{u}$ such that $\Hom_G(k,l)=\Hom(k,l)$ for all $k,l\in\N_0$. It is universal in the sense that whenever $(H,v)$ is a compact matrix quantum group such that $Tv^{\otimes k}=v^{\otimes l}T$ for all $T\in\Hom(k,l)$ and all $k,l$, then $H$ is a quantum subgroup of $G$.
\end{thm}

We conclude that compact matrix quantum groups are determined by their intertwiner spaces, hence all we need to know about a compact matrix quantum group $(G,u)$ is $(\Hom_G(k,l))_{k,l\in\N_0}$.

\subsection{Linear maps associated to partitions}
\label{SectLinMaps}

Banica and Speicher \cite{BS} associated linear maps to partitions $p\in P$  in order to obtain quantum groups whose intertwiner spaces are of a combinatorial form. 
Let $n\in\N$ and $p\in P(k,l)$. Let $i=(i_1,\ldots,i_k)$ and $j=(j_1,\ldots,j_l)$ be multi indices whose components range in $\{1,\ldots,n\}$. We decorate the $k$ upper points of $p$ from left to right with the entries of $i$, and likewise for the $l$ lower points (from left to right)  using $j$. If the strings of $p$ connect only equal indices, then $\delta_p(i,j):=1$ and $\delta_p(i,j):=0$ otherwise. See \cite[Ex. 4.2]{TW} or \cite{VSW} for examples.

Let $e_1,\ldots, e_n$ be the canonical basis of $\C^n$. 
For $p\in P(k,l)$, we define the linear map $T_p:(\C^n)^{\otimes k}\to(\C^n)^{\otimes l}$ by setting
\[T_p(e_{i_1}\otimes \cdots \otimes e_{i_k}):=\sum_{j_1,\ldots, j_l=1}^n \delta_p(i,j)e_{j_1}\otimes \cdots \otimes e_{j_l}.\]
The convention is to put $(\C^n)^{\otimes 0}=\C$.
The maps $T_p$ behave nicely with respect to the category operations.

\begin{prop}[{\cite{BS}}]\label{PropTpBS}
We have:
\begin{itemize}
\item[(a)] $T_p\otimes T_q=T_{p\otimes q}$,
\item[(b)]  $(T_p)^*=T_{p^*}$,
\item[(c)] $T_qT_p=n^{b(q,p)}T_{qp}$, where $b(q,p)$ is the number of disconnected strings arising in the composition of $p$ and $q$,
\item[(d)] $T_{\idpart}=\id:\C^n\to\C^n$,
\item[(e)] $T_{\paarpart}=\sum_i e_i\otimes e_i\in(\C^n)^{\otimes 2}$.
\end{itemize}
\end{prop}

\subsection{Linear maps associated to spatial partitions}
\label{SectLinMMaps}

We now extend Section \ref{SectLinMaps} to spatial partitions.
Let $m\in\N$ and let $n_1,\ldots,n_m\in\N$. By  $\ker (n_1,\ldots,n_m)$ we denote the unique partition $\pi$ in $P(m)$ with the property that $s$ and $t$ are in the same block of $\pi$ if and only if $n_s=n_t$.  Let $p\in\Pm(k,l)$ be $\ker(n_1,\ldots,n_m)$-graded, i.e.  the strings of  $p$ connect different levels only if the ``dimensions'' $n_i$ of these  levels coincide.
We put
\[[n_1\times\ldots\times n_m]:=\{1,\ldots,n_1\}\times\{1,\ldots,n_2\}\times\ldots\times\{1,\ldots,n_m\}.\]
Let $I$ be a multi index in $[n_1\times\ldots\times n_m]^k$ and $J$ be a multi index in $[n_1\times\ldots\times n_m]^l$. Hence $I$ is of the form
\[I=(I_1,\ldots,I_k)=\left((i_1^1,\ldots,i_m^1),\ldots,(i_1^k,\ldots,i_m^k)\right).\]
 We define $\delta_p(I,J)$ as before, decorating the upper plane of $p$ by $I$ and the lower plane by $J$, i.e. under the identification $\Pm(k,l)\cong P(km,lm)$, we simply apply the former definition of $\delta_p$. We may find a natural orthonormal basis of $\C^{n_1n_2\ldots n_m}$ using the following isomorphism:
\begin{align*}
\C^{n_1}\otimes\ldots\otimes\C^{n_m}&\cong \C^{n_1n_2\ldots n_m}\\
e_{i_1}\otimes\ldots \otimes e_{i_m}&\longleftrightarrow  e_{(i_1,\ldots,i_m)}.
\end{align*}
We assign the following linear map $S_p$ to $p$:
\begin{align*}
S_p:(\C^{n_1n_2\ldots n_m})^{\otimes k}&\to(\C^{n_1n_2\ldots n_m})^{\otimes l}\\
e_{(i_1^1,\ldots,i_m^1)}\otimes\ldots\otimes e_{(i_1^k,\ldots,i_m^k)}&\mapsto\sum_{j_1^1,\ldots,j_m^1,\ldots,j_1^l,\ldots,j_{m}^l}\delta_p(I,J)e_{(j_1^1,\ldots,j_m^1)}\otimes\ldots\otimes e_{(j_1^l,\ldots,j_m^l)}.
\end{align*}
For $m=1$ and $p\in P^{(1)}(k,l)=P(k,l)$, the constructions of $T_p$ and $S_p$ coincide.

\begin{rem}\label{RemKey}
The definition of $S_p$ constitutes the technical key observation of this article. It looks quite simple, but let us discuss it from a different perspective. Observe that the map $T_p:(\C^{n_1\cdots n_m})^{\otimes k}\to(\C^{n_1\cdots n_m})^{\otimes l}$ for $p\in P(k,l)$ coincides with $S_{p^{(m)}}:(\C^{n_1\cdots n_m})^{\otimes k}\to(\C^{n_1\cdots n_m})^{\otimes l}$ for $p^{(m)}\in \Pm(k,l)$. Hence, 
with the maps $S_p$ for general spatial partitions $p\in\Pm$, we can go ``finer'' than $T_p$, making use of  the decomposition of $n_1\cdots n_m$ into factors.
Since not all spatial partitions in $\Pm$ come from amplifications, the assignment $p\mapsto S_p$ is richer than the assignment $p\mapsto T_p$.
\end{rem}

Proposition \ref{PropTpBS} translates to the following.
\begin{prop}\label{PropSp}
We have:
\begin{itemize}
\item[(a)] $S_p\otimes S_q=S_{p\otimes q}$,
\item[(b)] $(S_p)^*=S_{p^*}$,
\item[(c)] $S_qS_p=(n_1\cdots n_m)^{b(q,p)}S_{qp}$,
\item[(d)] $S_{\idpart^{(m)}}=\id:\C^{n_1n_2\ldots n_m}\to\C^{n_1n_2\ldots n_m}$,
\item[(e)] $S_{\paarpart^{(m)}}=\sum_{(i_1,\ldots,i_m)}e_{(i_1,\ldots,i_m)}\otimes e_{(i_1,\ldots,i_m)}\in(\C^{n_1n_2\ldots n_m})^{\otimes 2}$.

\end{itemize}
\end{prop}
\begin{proof}
Let $p\in \Pm(k,l)$. Viewing it as a partition in $P(km,lm)$, we may assign the following map to it:
\begin{align*}
T_p:(\C^{n_1}\otimes\ldots\otimes\C^{n_m})^{\otimes k}&\to(\C^{n_1}\otimes\ldots\otimes\C^{n_m})^{\otimes l}\\
e_{i_1}\otimes\ldots\otimes e_{i_{km}}&\mapsto\sum_{j_1,\ldots,j_{lm}}\delta_p(I,J)e_{j_1}\otimes\ldots\otimes e_{j_{lm}}.
\end{align*}
Under the isomorphism 
\begin{align*}
\C^{n_1}\otimes\ldots\otimes\C^{n_m}&\cong \C^{n_1n_2\ldots n_m}\\
e_{i_1}\otimes\ldots \otimes e_{i_m}&\longleftrightarrow  e_{(i_1,\ldots,i_m)},
\end{align*}
it coincides with the map
\begin{align*}
S_p:(\C^{n_1n_2\ldots n_m})^{\otimes k}&\to(\C^{n_1n_2\ldots n_m})^{\otimes l}\\
e_{(i_1^1,\ldots,i_m^1)}\otimes\ldots\otimes e_{(i_1^k,\ldots,i_m^k)}&\mapsto\sum_{j_1^1,\ldots,j_m^1,\ldots,j_1^l,\ldots,j_{m}^l}\delta_p(I,J)e_{(j_1^1,\ldots,j_m^1)}\otimes\ldots\otimes e_{(j_1^l,\ldots,j_m^l)}.
\end{align*}
Thus, the assertions (a), (b) and (c) follow directly from Proposition \ref{PropTpBS}. The assertions (d) and (e) follow from Remark \ref{RemKey}.
\end{proof}

\subsection{Definition of spatial partition quantum groups}\label{SectMeasy}

The properties of Proposition \ref{PropSp} ensure that  $\lspan\{S_p\;|\; p\in \CC(k,l)\}$ is an abstract $C^*$-tensor category in Woronowicz's sense. Hence we may apply the Tannaka-Krein Theorem to it in order to obtain a quantum group. This motivates the following definition generalizing the one by Banica and Speicher \cite{BS}. See also \cite{VSW, Web, RW} for more on Banica-Speicher quantum groups (easy quantum groups) and \cite{TW} (or rather the appendix of the arXiv version of \cite{TW}) for an explicit transition from categories of partitions to quantum groups via Tannaka-Krein.

\begin{defn}
Let $m\in\N$ and let $n_1,\ldots,n_m\in\N$. A compact matrix quantum group $(G,u)$ with $G\subset O_{n_1\cdots n_m}^+$ and $u\in M_{n_1\cdots n_m}(C(G))$ is a \emph{spatial partition quantum group}, if there is a category of $\ker(n_1,\ldots,n_m)$-graded partitions $\CC\subset \Pm$ such that for all $k,l\in\N_0$, the intertwiner spaces of $G$ are of the form
\[\Hom_G(k,l)=\lspan\{S_p\;|\; p\in \CC(k,l)\}.\]
\end{defn}

It is convenient to use multi indices from $[n_1\times\ldots\times n_m]$ for the matrix $u\in M_{n_1\cdots n_m}(C(G))$, i.e. $u=(u_{IJ})_{I,J\in [n_1\times\ldots\times n_m]}$.
Banica and Speicher \cite{BS} defined Banica-Speicher quantum groups using the maps $T_p$ for $p\in P$. For $m=1$, their quantum groups and our spatial partition quantum groups coincide.

\subsection{$C^*$-algebraic relations associated to spatial partitions}
\label{SectRelations}

The equations $S_pu^{\otimes k}=u^{\otimes l}S_p$, for $p\in\Pm(k,l)$ give rise to relations on the $u_{IJ}$. They are the following.
\begin{defn}
Let $m\in\N$, $n_1,\ldots,n_m\in\N$ and let $p\in \Pm(k,l)$. We say that elements $u_{IJ}$, $I,J\in [n_1\times\ldots\times n_m]$ \emph{satisfy the relations $R(p)$}, if, for all choices of multi indices $I=(I_1,\ldots, I_k)\in [n_1\times\ldots\times n_m]^k$ and $J=(J_1,\ldots, J_l)\in [n_1\times\ldots\times n_m]^l$, we have
\[\sum_{A_1,\ldots,A_k\in [n_1\times\ldots\times n_m]} \delta_p(A,J) u_{A_1I_1}\ldots u_{A_kI_k}=\sum_{B_1,\ldots,B_l\in [n_1\times\ldots\times n_m]} \delta_p(I,B) u_{J_1B_1}\ldots u_{J_lB_l}.\]
\end{defn}

\begin{lem}\label{LemRelations}
We have $S_pu^{\otimes k}=u^{\otimes l}S_p$ if and only if the relations $R(p)$ are satisfied.
\end{lem}
\begin{proof}
Using the matrix units $e_{JI}\in M_{n_1\cdots n_m}(\C)$ for $J,I\in [n_1\times\ldots\times n_m]$, we write
\[u^{\otimes k}=\sum_{\substack{I_1,\ldots,I_k \\ J_1,\ldots,J_k}} e_{J_1I_1}\otimes\ldots\otimes e_{J_kI_k}\otimes u_{J_1I_1}\cdots u_{J_kI_k}\in M_{n_1\cdots n_m}(\C)\otimes \ldots\otimes M_{n_1\cdots n_m}(\C)\otimes C(G).\]
Applying it to a vector $e_{I_1}\otimes \ldots\otimes e_{I_k}\otimes 1$, we obtain
\[u^{\otimes k}(e_{I_1}\otimes \ldots\otimes e_{I_k}\otimes 1)=\sum_{A_1,\ldots,A_k} e_{A_1}\otimes\ldots\otimes e_{A_k}\otimes u_{A_1I_1}\cdots u_{A_kI_k}.\]
Thus
\begin{align*}
&S_pu^{\otimes k}(e_{I_1}\otimes\ldots\otimes e_{I_k}\otimes 1)=\sum_{J_1,\ldots,J_k} e_{J_1}\otimes\ldots\otimes e_{J_l}\otimes\left(\sum_{A_1,\ldots,A_k} \delta_p(A,J) u_{A_1I_1}\cdots u_{A_kI_k}\right),\\ 
&u^{\otimes k}S_p(e_{I_1}\otimes\ldots\otimes e_{I_k}\otimes 1)=\sum_{J_1,\ldots,J_k} e_{J_1}\otimes\ldots\otimes e_{J_l}\otimes\left(\sum_{B_1,\ldots,B_l} \delta_p(I,B) u_{J_1B_1}\cdots u_{J_lB_l}\right).
\end{align*}
Hence, $S_pu^{\otimes k}=u^{\otimes l}S_p$ if and only if the relations $R(p)$ hold.
\end{proof}

\begin{prop}\label{PropTKeasy}
Let $G\subset O_{n_1\ldots n_m}^+$ be a compact matrix quantum group. If $C(G)$ is the universal unital $C^*$-algebra generated by self-adjoint elements $u_{IJ}$ such that $u$ is orthogonal and the relations $R(p)$ are satisfied for all $p\in\CC$ for some $\ker(n_1,\ldots,n_m)$-graded category $\CC\subset \Pm$, then $G$ is a spatial partition quantum group.

In particular, if $\CC=\langle p_1,\ldots, p_k\rangle$, then the relations $R(p)$ are satisfied for all $p\in\CC$  if and only if the relations $R(p_1), \ldots, R(p_k)$ and $R(\paarpart^{(m)})$ are satisfied.
\end{prop}
\begin{proof}
The proof is similar to \cite[Prop. 5.7]{TW} (see also the appendix of the arXiv version of \cite{TW}): The space $W:=\lspan\{S_p\;|\; p\in\CC\}$ is a $W^*$-tensor category in the sense of Woronowicz; hence we may associate a quantum group $H$ to it and the generators of $C(H)$ satisfy all relations $R(p)$, by Lemma \ref{LemRelations}. We thus have a surjection from $C(G)$ to $C(H)$. Conversely, $C(G)$ is a model of $W$ which yields a map from $C(H)$ to $C(G)$ by universality of $H$. Hence $G=H$.

It is a direct algebraic computation to check that the relations $R(p\otimes q)$, $R(pq)$ and $R(p^*)$ hold, whenever $R(p)$ and $R(q)$ hold. Thus, the relations $R(p_1), \ldots, R(p_k)$ and $R(\paarpart^{(m)})$ imply the relations $R(p)$ for all $p\in\CC$.
\end{proof}

Here is a list of the relations associated to the generators of $P^{(2)}$ and $P_2^{(2)}$.

\begin{align*}
\idpart^{(2)} &:
u_{(i_1i_2)(j_1j_2)}=u_{(i_1i_2)(j_1j_2)}\\
\paarpart^{(2)}&, {\paarpart^{(2)}}^* :
\sum_{g_1,g_2} u_{(i_1,i_2)(g_1,g_2)}u_{(j_1,j_2)(g_1,g_2)}
=\sum_{g_1,g_2} u_{(g_1,g_2)(i_1,i_2)}u_{(g_1,g_2)(j_1,j_2)}
=\delta_{i_1j_1}\delta_{i_2j_2}\\
\FourOneIdTwopart &:
u_{(k_1,k_2)(i_1,i_2)}u_{(k_3,k_4)(i_1,i_3)}=0 \textnormal{ if } k_1\neq k_3;\qquad
u_{(i_1,i_2)(k_1,k_2)}u_{(i_1,i_3)(k_3,k_4)}=0 \textnormal{ if } k_1\neq k_3\\
\IdOneFourTwopart &:
u_{(k_1,k_2)(i_1,i_2)}u_{(k_3,k_4)(i_3,i_2)}=0 \textnormal{ if } k_2\neq k_4;\qquad
u_{(i_1,i_2)(k_1,k_2)}u_{(i_3,i_2)(k_3,k_4)}=0 \textnormal{ if } k_2\neq k_4
\end{align*}
\begin{align*}
\SingSingOneIdTwopart &:
\sum_g u_{(g,b_2)(i_1,i_2)}=\sum_h u_{(b_1,b_2)(h,i_2)} \textnormal{ (in particular independent of } i_1,b_1)\\
\IdOneSingSingTwopart &:
\sum_g u_{(b_1,g)(i_1,i_2)}=\sum_h u_{(b_1,b_2)(i_1,h)} \textnormal{ (in particular independent of } i_2,b_2)\\
\SingOneSingTwopart&, \SingOneSingTwopart {}^* :
\sum_{g_1,g_2} u_{(i_1,i_2)(g_1,g_2)}
=\sum_{g_1,g_2} u_{(g_1,g_2)(j_1,j_2)}
=1\\
\CrossOneIdTwopart &:
u_{(b_1,b_2)(i_1,i_2)}u_{(b_3,b_4)(i_3,i_4)}
=u_{(b_3,b_2)(i_3,i_2)}u_{(b_1,b_4)(i_1,i_4)}\\
\IdOneCrossTwopart &:
u_{(b_1,b_2)(i_1,i_2)}u_{(b_3,b_4)(i_3,i_4)}
=u_{(b_1,b_4)(i_1,i_4)}u_{(b_3,b_2)(i_3,i_2)}\\
\PaarPaarOneIdTwopart &:
\delta_{j_1j_3}\sum_gu_{(g,j_2)(i_1,i_2)}u_{(g,j_4)(i_3,i_4)}=\delta_{i_1i_3}\sum_hu_{(j_1,j_2)(h,i_2)}u_{(j_3,j_4)(h,i_4)}
\\
\IdOnePaarPaarTwopart &:
\delta_{j_2j_4}\sum_gu_{(j_1,g)(i_1,i_2)}u_{(j_3,g)(i_3,i_4)}=\delta_{i_2i_4}\sum_hu_{(j_1,j_2)(i_1,h)}u_{(j_3,j_4)(i_3,h)}\\
\FourOneTwopart &:
\delta_{b_1b_2}\sum_{g} u_{(g,g)(i_1,i_2)}
=\delta_{i_1i_2}\sum_{h} u_{(b_1,b_2)(h,h)}\\
\PaarOneTwopart &:
\sum_{g} u_{(b_1,b_2)(g,g)}=\delta_{b_1b_2}\\
\PaarPaarOneTwopart &:
\delta_{i_1i_2}\sum_gu_{(g,g)(j_1,j_2)}=\delta_{j_1j_2}\sum_{g} u_{(i_1,i_2)(g,g)}\\
\CrossOneTwopart &:
u_{(i_1,i_2)(j_1,j_2)}=u_{(i_2,i_1)(j_2,j_1)}
\end{align*}

\begin{rem}
Inspired from the above relations for $\PaarOneTwopart$, we view them more generally for $(u_{ij})_{i,j=1,\ldots,n}$ as
\[\sum_{k\in I} u_{ik}=\delta_{i\in I}\]
for some subset $I\subset\{1,\ldots,n\}$. For instance:
\[\sum_{k\textnormal{ even}} u_{ik}=\delta_{i\textnormal{ even}}.\]
It is easy to check that these relations pass through the comultiplication. Hence, one can define some partial versions of quantum permutation groups.
\end{rem}

\section{Products of categories}
\label{SectProducts}

Given two categories of partitions $\CC_1$ and $\CC_2$ -- how can we form a new one from this data? Several possibilities will be developped in the sequel.

\subsection{Kronecker product of categories}

In the setting of spatial partition quantum groups, we have an obvious possibility to form a new category out of two given categories $\CC_1\subset P$ and $\CC_2\subset P$: We simply put $\CC_1$ on level one and $\CC_2$ on level two. More generally, we have the following setup.

\begin{defn}\label{DefKronecker}
Let $s\in\N$. Let $m_i\in\N$ and let $\pi_i\in P(m_i)$ for $i=1,\ldots,s$.
\begin{itemize}
\item[(a)] Let $k,l\in\N_0$. Let $p_i\in P^{(m_i)}_{\pi_i}(k,l)$ for $i=1,\ldots,s$. We denote by
\[\begin{pmatrix}p_s\\\vdots\\p_1\end{pmatrix}\in P^{(m_1+\ldots+m_s)}_{\pi_1\otimes\ldots\otimes \pi_s}(k,l)\]
the partition given by placing $p_1$ on the levels $1$ to $m_1$, placing $p_2$ on the levels $m_1+1$ to $m_1+m_2$, and so on.
\item[(b)] Let $\CC_i\subset P^{(m_i)}_{\pi_i}$ be sets of $\pi_i$-graded partitions, for $i=1,\ldots,s$. We denote by
\[\CC_1\times\ldots\times\CC_s:=\{ \begin{pmatrix}p_s\\\vdots\\p_1\end{pmatrix}\in P^{(m_1+\ldots+m_s)}_{\pi_1\otimes\ldots\otimes \pi_s}\;|\; p_i\in\CC_i \textnormal{ for all }i=1,\ldots,s\}\subset P^{(m_1+\ldots+m_s)}_{\pi_1\otimes\ldots\otimes \pi_s}\]
the \emph{Kronecker product} of the sets $\CC_i$, $i=1,\ldots,s$.
\end{itemize}
\end{defn}

\begin{lem}
Let $\CC_i\subset P^{(m_i)}_{\pi_i}$ be categories of $\pi_i$-graded spatial partitions, for $i=1,\ldots,s$. Then $\CC_1\times\ldots\times\CC_s\subset P^{(m_1+\ldots+m_s)}_{\pi_1\otimes\ldots\otimes \pi_s}$ is category of $\pi_1\otimes\ldots\otimes \pi_s$-graded  spatial partitions.
\end{lem}
\begin{proof}
The proof is straightforward.
\end{proof}

\subsection{Glued tensor products of spatial partition quantum groups}

It is natural to ask for the quantum group picture of the above Kronecker product of categories. Recall the following product of quantum groups from \cite[Def. 6.4]{TW}.

\begin{defn}\label{DefGlued}
Let $(G,u)$ and $(H,v)$ be two compact matrix quantum groups with $u=(u_{ij})_{i,j=1,\ldots,n}$ and $v=(v_{kl})_{k,l=1,\ldots,m}$. The \emph{glued direct product} $G\tilde\times H$ of $G$ and $H$ is given by the $C^*$-subalgebra
\[C(G\tilde\times H):=C^*(u_{ij}v_{kl}\;|\; i,j=1,\ldots,n\textnormal{ and } k,l=1,\ldots m)\subset C(G)\otimes_{\textnormal{max}} C(H).\]
Here, we identify $C(G)\otimes_{\textnormal{max}} C(H)$ with the universal $C^*$-algebra generated by elements $u_{ij}\in C(G)$ and $v_{kl}\in C(H)$ such that all $u_{ij}$ commute with all $v_{kl}$.
\end{defn}

\begin{thm}\label{ThmProduct}
Let $(G_i,u_i)\subset O_{N_i}^+$ be spatial partition quantum groups with categories $\CC_i\subset P^{(m_i)}$ for $i=1,2$ and $N_i=n_1^i\cdots n_{m_i}^i$. The spatial partition quantum group associated to the category $\CC_1\times \CC_2$ is $G_1\tilde\times G_2\subset O_{N_1N_2}^+$.
\end{thm}
\begin{proof}
It is straightforward to check that $G_1\tilde\times G_2$ is indeed a quantum subgroup of $O_{N_1N_2}^+$ in the sense of Definition \ref{DefSubgroup}(a) mapping the generators $w_{(i_1,i_2)(j_1,j_2)}$ of $O_{N_1N_2}^+$ to $u_{i_1j_1}v_{i_2j_2}\in C(G_1\tilde\times G_2)$.
Moreover, let $G$ be the compact matrix quantum group associated to the category $\CC_1\times \CC_2$. Thus its intertwiner space is
\[\Hom_{G}(k,l)=\lspan\{S_p\;|\; p\in \CC_1\times \CC_2\}.\]
By definition, the intertwiner space of $G_1\tilde\times G_2$ is
\[\Hom_{G_1\tilde\times G_2}(k,l)=\{T:(\C^{N_1}\otimes\C^{N_2})^{\otimes k}\to(\C^{N_1}\otimes\C^{N_2})^{\otimes l}\;|\; T(u_1\otimes u_2)^{\otimes k}=(u_1\otimes u_2)^{\otimes l}T\}.\]
To prove the statement of the theorem, it suffices to prove that
\[\Hom_{G}(k,l)=\Hom_{G_1\tilde\times G_2}(k,l)\qquad\textnormal{for all }k,l\in\N_0.\]
For doing so, let us first consider $S_p\in \Hom_{G}(k,l)$. Thus $p=\begin{pmatrix}p_2\\p_1\end{pmatrix}$ with $p_i\in\CC_i$. Reordering the elements of the tensor product $(\C^{N_1}\otimes\C^{N_2})^{\otimes k}\cong(\C^{N_1})^{\otimes k}\otimes(\C^{N_2})^{\otimes k}$, we observe that $S_p\cong S_{p_1}\otimes S_{p_2}$. Thus
\[S_p(u_1\otimes u_2)^{\otimes k}\cong(S_{p_1}\otimes S_{p_2}) (u_1^{\otimes k}\otimes u_2^{\otimes k})= (u_1^{\otimes k}\otimes u_2^{\otimes k})(S_{p_1}\otimes S_{p_2})\cong(u_1\otimes u_2)^{\otimes k}S_p.\]
Consequently, $S_p\in \Hom_{G_1\tilde\times G_2}(k,l)$ and by linearity, $\Hom_{G}(k,l)\subset \Hom_{G_1\tilde\times G_2}(k,l)$.

We conclude by a dimension argument. Recall that the dimension of $\Hom_H(k,l)$ of a compact matrix quantum group $(H,w)$ is given by $h_H(\chi_w^{k+l})$ where $h_H$ is the Haar measure on $H$ and $\chi_w=\sum_i w_{ii}$. The Haar measure of $G_1\tilde\times G_2$ is given by $h_{G_1}\otimes h_{G_2}$ by \cite{WanTens}. We have
\begin{align*}
\dim\Hom_{G_1\tilde\times G_2}(k,l)
&= h_{G_1}\otimes h_{G_2}(\chi_{u_1\otimes u_2}^{k+l})\\
&= h_{G_1}\otimes h_{G_2}(\chi_{u_1}^{k+l}\chi_{u_2}^{k+l})\\
&= h_{G_1}(\chi_{u_1}^{k+l}) h_{G_2}(\chi_{u_2}^{k+l})\\
&=\dim\Hom_{G_1}(k,l)\cdot\dim\Hom_{G_2}(k,l)\\
&=\dim \lspan\{S_{p_1}\;|\;p_1\in \CC_1\}\cdot \dim \lspan\{S_{p_2}\;|\;p_2\in \CC_2\}\\
&=\dim \lspan\{S_{p_1}\;|\;p_1\in \CC_1\}\otimes \lspan\{S_{p_2}\;|\;p_2\in \CC_2\}\\
&=\dim \lspan\{S_{p_1}\otimes S_{p_2}\;|\;p_1\in \CC_1,p_2\in \CC_2\}.
\end{align*}
Once again, reordering the elements of the tensor product $(\C^{N_1}\otimes\C^{N_2})^{\otimes k}\cong(\C^{N_1})^{\otimes k}\otimes(\C^{N_2})^{\otimes k}$, we observe that $S_{p_1}\otimes S_{p_2}\cong S_p$ whenever $p=\begin{pmatrix}p_2\\p_1\end{pmatrix}$. It allows us to conclude:
\begin{align*}
\dim\Hom_{G_1\tilde\times G_2}(k,l)
&=\dim \lspan\{S_{p_1}\otimes S_{p_2}\;|\;p_1\in \CC_1,p_2\in \CC_2\}\\
&=\dim \lspan\{S_{p}\;|\;p=\begin{pmatrix}p_2\\p_1\end{pmatrix},p_1\in \CC_1,p_2\in \CC_2\}\\
&=\dim\lspan\{S_p\;|\; p\in \CC_1\times \CC_2\}\\
&=\dim\Hom_{G}(k,l).
\end{align*}
\end{proof}

\begin{cor}
In particular, the quantum groups $S_n\tilde\times S_n$, $S_n\tilde\times S_n^+$, $S_n^+\tilde\times S_n^+$ or $O_n\tilde\times S_n^+$ etc. are spatial partition quantum groups.
\end{cor}

We observe that the class of spatial partition quantum groups is closed under taking the glued direct product (increasing the number $m$). This is not the case for Banica-Speicher quantum groups -- taking the glued direct product, we leave the class of Banica-Speicher quantum groups entering the class of spatial partition quantum groups.

\subsection{Glued tensor products with amalgamation over partitions}

Forming the Kronecker product $\CC_1\times \CC_2$ of two categories $\CC_i\subset P$ (with Banica-Speicher quantum groups $(G,u)$ and $(H,v)$), we obtain a category of spatial partitions respecting the (two) levels. We will obtain further relations for the generators $u_{ij}v_{kl}$ of $C(G\tilde\times H)$, if we throw in partitions mixing the levels.

\begin{defn}\label{DefAmalgam}
Let $(G,u)$ and $(H,v)$ be compact matrix quantum groups such that the matrices $u$ and $v$ have the same size. Put $w_{(i,k)(j,l)}:=u_{ij}v_{kl}\in C(G)\otimes_{\textnormal{max}} C(H)$.
Let $p\in P^{(2)}$. 
The \emph{glued direct product over $p$} (or the \emph{$p$-amalgamated glued direct product} $G\tilde\times_p H$ of $G$ and $H$) is given by the $C^*$-subalgebra
\begin{align*}
C^*(u_{ij}v_{kl}\;|\; &i,j, k, l=1,\ldots,n)\\
&\subset C(G)\otimes_{\textnormal{max}} C(H) / \langle \textnormal{the elements $w_{(i,k)(j,l)}$ satisfy the relations } R(p)\rangle.
\end{align*}
\end{defn}

\begin{lem}
The $C^*$-algebra in Definition \ref{DefAmalgam} admits a comultiplication turning $G\tilde\times_p H$ into a compact matrix quantum group with fundamental representation $(u_{ij}v_{kl})_{i,j,k,l=1,\ldots,n}$.
\end{lem}
\begin{proof}
The $C^*$-algebra $C(G\tilde\times H)$ of Definition \ref{DefGlued} admits a comultiplication due to \cite{TW}. We can view $C(G\tilde\times_p H)$ as a quotient of $C(G\tilde\times H)$, and consider the following diagram.
\begin{align*}
&C(G\tilde\times H)&&\stackrel{\Delta}{\longrightarrow}&&C(G\tilde\times H)\otimes C(G\tilde\times H)\\
&\alpha\downarrow && &&\downarrow\alpha\otimes\alpha\\
&C(G\tilde\times_p H)&& &&C(G\tilde\times_p H)\otimes C(G\tilde\times_p H)
\end{align*}
Hence, all we have to check is that the map $(\alpha\otimes\alpha)\circ\Delta$ factorizes through $\alpha$. For doing so, we only need to check that the elements
\[\sum_{s,t}w_{(i,k)(s,t)}\otimes w_{(s,t)(j,l)}\in C(G\tilde\times_p H)\otimes C(G\tilde\times_p H)\]
satisfy the relations $R(p)$, which is the case.
\end{proof}

\begin{thm}\label{ThmAmalProd}
Let $(G_i,u_i)\subset O_{n}^+$ be Banica-Speicher quantum groups with categories $\CC_i\subset P$ for $i=1,2$. Let $p\in P^{(2)}$. The spatial partition quantum group associated to the category $\langle\CC_1\times \CC_2,p\rangle$ is $G_1\tilde\times_p G_2\subset O_{n^2}^+$.
\end{thm}
\begin{proof}Let $(G,v)$ be the spatial partition quantum group associated to the category $\langle\CC_1\times \CC_2,p\rangle$. We denote by $u$ the generating matrix of $G_1\tilde\times_p G_2$ and by $w$ the generating matrix of $G_1\tilde\times G_2$. Because the relations $R(p)$ only involve elements of $C(G_1\tilde\times G_2)$, we have$$C(G_1\tilde\times_p G_2)\cong C(G_1\tilde\times G_2)/ \langle \textnormal{the elements $w_{(i,k)(j,l)}$ satisfy the relations } R(p)\rangle$$via  the canonical isomorphism $w_{IJ}\mapsto u_{IJ}$.

By definition, the space of intertwiners of $G$ contains the one of $G_1\tilde\times G_2$. By universality of Tannaka-Krein theorem, it means that $G$ is a quantum subgroup of $G_1\tilde\times G_2$, or more precisely, that there exists a surjective $*$-homomorphism $\phi:C(G_1\tilde\times G_2)\to C(G)$ mapping $w_{IJ}$ to $v_{IJ}$. Because $v_{IJ}$ satisfy the relation $R(p)$, this homomorphism can be quotiented into a surjective $*$-homomorphism $\phi:C(G_1\tilde\times_p G_2)\to C(G)$ mapping $u_{IJ}$ to $v_{IJ}$, meaning that $G$ is a quantum subgroup of $G_1\tilde\times_p G_2$.

Conversely, the space of intertwiners of $G_1\tilde\times_p G_2$ is bigger than the one of $G$, because it contains $S_p$ and $S_q$ for $q\in \CC_1\times \CC_2$, which means that it contains $\lspan\{S_q\;|\; q\in \langle\CC_1\times \CC_2,p\rangle\}$. By universality of Tannaka-Krein theorem, it implies that $G_1\tilde\times_p G_2$ is a quantum subgroup of $G$.
\end{proof}

It is straightforward to generalize Definition \ref{DefAmalgam} to products of an arbitrary finite number $m\in\N$ of quantum groups, allowing for an amalgamation with arbitrary partitions $P^{(m)}$. One can also choose $G_1$ and $G_2$ of the above proposition to be spatial partition quantum groups rather than Banica-Speicher quantum groups, after extending Definition \ref{DefAmalgam} to an amalgamation of $G\subset O_{n_1^1\cdots n_{m_1}^1}^+$ and $H\subset O_{n_1^2\cdots n_{m_2}^2}^+$ with respect to a partition $p\in P^{(m_1+m_2)}_\pi$, where $\pi=\ker(n_1^1,\ldots,n_{m_1}^1, n_1^2,\ldots,n_{m_2}^2)$. Observe that an amalgamation with a partition respecting the levels boils down to the glued direct product without amalgamation.

\begin{ex}\label{ExAmalgam}
Let $\CC_1=\CC_2=NC_2$.
\begin{itemize}
\item[(a)] For $p=\PaarOneTwopart$, the category $\langle NC_2\times NC_2,\PaarOneTwopart\rangle$ corresponds to $O_n^+\tilde\times_p O_n^+$ with
\[C(O_n^+\tilde\times_p O_n^+)=C^*(u_{ij}v_{kl}\;|\; (u_{ij}), (v_{kl}) \textnormal{ are orth.}, u_{ij}v_{kl}=v_{kl}u_{ij}, \sum_k u_{i_1k}v_{i_2k}=\delta_{i_1i_2}).\]
\item[(b)] For $p=\CrossOneTwopart$, the category $\langle NC_2\times NC_2,\CrossOneTwopart\rangle$ corresponds to $O_n^+\tilde\times_p O_n^+$ with
\[C(O_n^+\tilde\times_p O_n^+)=C^*(u_{ij}v_{kl}\;|\; (u_{ij}), (v_{kl}) \textnormal{ are orth.}, u_{ij}v_{kl}=v_{kl}u_{ij},u_{ij}v_{kl}=u_{kl}v_{ij}).\]
\end{itemize}
\end{ex}

\section{Examples of spatial partition quantum groups}
\label{SectExamples}

The new machine of spatial partition quantum groups provides many new examples of compact matrix quantum groups, some of which are presented in this section. Before doing so, we take a look at the natural cornerstones of the theory.

\subsection{Amplifications of Banica-Speicher quantum groups}

A trivial class of spatial partition quantum groups is obtained by the amplification of Banica-Speicher quantum groups.

\begin{prop}\label{PropAmplif}
Let $n_1,\ldots,n_m\in\N$. Let $G_n\subset O_n^+$ be a Banica-Speicher quantum group with category $\CC\subset P$. Then, the spatial partition quantum group associated to $[\CC]^{(m)}$ is the Banica-Speicher quantum group $G_{n_1\cdots n_m}\subset O_{n_1\cdots n_m}^+$ with category $\CC$.
\end{prop}
\begin{proof}
By Lemma \ref{LemAmplif}, $[\CC]^{(m)}\subset \Pm$ is a category of spatial partitions; it thus gives rise to a spatial partition quantum group. For $p\in\CC(k,l)$,  the maps
\[T_p: (\C^{n_1n_2\ldots n_m})^{\otimes k}\to (\C^{n_1n_2\ldots n_m})^{\otimes l}
\quad\textnormal{ and }\quad
S_{p^{(m)}}: (\C^{n_1n_2\ldots n_m})^{\otimes k}\to (\C^{n_1n_2\ldots n_m})^{\otimes l}\]
coincide by Remark \ref{RemKey}.
As a consequence, the intertwiner spaces of the spatial partition quantum group associated to $[\CC]^{(m)}$ and of the Banica-Speicher quantum group with category $\CC$ are the same, which allows us to conclude by the Tannaka-Krein theorem.
\end{proof}

\begin{cor}
We have the following correspondences of spatial partition quantum groups:
\begin{center}
\begin{tabular}{ccccccc}
$S_{n_1\cdots n_m}^+$ & $\longleftrightarrow$ & $[NC]^{(m)}$,
&\qquad
$O_{n_1\cdots n_m}^+$ & $\longleftrightarrow$ & $[NC_2]^{(m)}$,\\
$S_{n_1\cdots n_m}$ & $\longleftrightarrow$ & $[P]^{(m)}$,
&\qquad
$O_{n_1\cdots n_m}$ & $\longleftrightarrow$ & $[P_2]^{(m)}$.
\end{tabular}
\end{center}
\end{cor}

\subsection{Minimal and maximal spatial partition quantum groups}

In the case $m=1$, we have $S_n\longleftrightarrow P$ and $O_n^+\longleftrightarrow NC_2$. Since every category of partitions satisfies $P\supset \CC\supset NC_2$, we have
\[S_n\subset G\subset O_n^+\]
 for Banica-Speicher quantum groups $G$. The case $m\geq 2$ is different, since we may have
\[S_{n_1\cdots n_m}\not\subset G\subset O_{n_1\cdots n_m}^+\]
for spatial partition quantum groups $G$, whenever $[P]^{(m)}\not\supset\CC\supset [NC_2]^{(m)}$ (see Example \ref{ExCateg}). We are thus interested in finding the minimal spatial partition quantum group corresponding to the maximal category of spatial partitions $\Pm$.

\begin{thm}\label{PropMin}
Let $n_1,\ldots,n_m\in\N$ and let $\pi=\ker(n_1,\ldots,n_m)$. The category $\Pm_\pi$ of all $\pi$-graded partitions corresponds to the spatial partition quantum group 
\[S_{n_{i_1}}\tilde\times\ldots \tilde\times S_{n_{i_r}}\subset O_{n_1\cdots n_m}^+,\]
 where $\{n_{i_1},\ldots, n_{i_r}\}$ is the set $\{n_1,\ldots,n_m\}$ without repetitions. 
 
 In the special case $n_1=\ldots =n_m=n$, we have
\[\Pm\quad\longleftrightarrow\quad S_n\subset O_{n^m}^+.\]
 
\end{thm}
\begin{proof}
We only prove the special case $m=2$ and $n_1=n_2=n$, the general case following from a straightforward adaption and an application of Theorem~\ref{ThmProduct}. Recall from Definition \ref{DefSubgroup} and Example \ref{ExCMQG} that $S_n$ can be viewed as a quantum subgroup of $O_{n^2}^+$ by mapping the generators $v_{(i_1,i_2)(j_1,j_2)}$ of $C(O_{n^2}^+)$ to the product $v'_{(i_1,i_2)(j_1,j_2)}:=u_{i_1j_1}u_{i_2j_2}$ in $C(S_n)$.

Let $A$ be the $C^*$-algebra generated by elements $v_{(i_1,i_2)(j_1,j_2)}$ satisfying all relations $(R_p)$ for all $p\in\Pm$. By Proposition \ref{PropTKeasy} and Corollary \ref{CorGenerators}, this is equivalent to satisfying all relations $(R_p)$ for all generators $p$ of $P^{(2)}$ as listed in Corollary \ref{CorGenerators}. It is easy to check that $v'_{(i_1,i_2)(j_1,j_2)}\in C(S_n)$ satisfies all these relations, hence a map $\phi:A\to C(S_n)$ mapping $v_{(i_1,i_2)(j_1,j_2)}\to v'_{(i_1,i_2)(j_1,j_2)}$ exists by the universal property. Conversely, the  elements $u_{ij}':=\sum_k v_{(ik)(j1)}\in A$ satisfy the relations of $C(S_n)$ as can be verified directly. This yields a map $\psi:C(S_n)\to A$ mapping $u_{ij}$ to $u_{ij}'$ by the universal property and we have that $\phi$ and $\psi$ are inverse to each other. Thus, $A$ and $C(S_n)$ are isomorphic; the isomorphism respects $\Delta$.

An alternative proof using intertwiners is based on the observation that the map from $C(O_{n^2}^+)$ to $C(S_n)$ maps the matrix $v$ to $u^{\otimes 2}$. Thus, intertwiners between $v^{\otimes k}$ and $v^{\otimes l}$ give rise to  intertwiners between $(u^{\otimes 2})^{\otimes k}=u^{\otimes 2k}$ and $(u^{\otimes 2})^{\otimes l}=u^{\otimes 2l}$. Since the linear span of $\{T_p\;|\;p\in P(2k,2l)\}$ coincides with the linear span of $\{S_p\;|\;p\in P^{(2)}(k,l)\}$, we deduce that the intertwiners of $S_n$ viewed as a subgroup of $O_{n^2}^+$ and the intertwiners of the spatial partition quantum group which corresponds to the category of all spatial partitions on two levels are the same, which allows us to conclude by the Tannaka-Krein theorem.
\end{proof}

We conclude that in the case $m=2$ and $n_1=n_2=n$, we have, for any spatial partition quantum group $G$,
\[S_n\subset G\subset O_{n^2}^+.\]
Recall that the class of Banica-Speicher quantum groups only covers the case $S_{n^2}\subset G\subset O_{n^2}^+$. 

\begin{rem}
Although our approach yields a larger class of quantum subgroups of $O_{n^2}^+$, we may not construct a quantum group $G$ with $S_{n^2}\subset G\subset O_{n^2}^+$ which is not a Banica-Speicher quantum group. Indeed, if $G$ is spatial partition with category $\CC\subset P^{(2)}$ and if $S_{n^2}\subset G\subset O_{n^2}^+$, then $\CC\subset [P]^{(2)}$ since $S_{n^2}$ corresponds to $[P]^{(2)}$. But this means that any partition $p\in\CC$ is a $2$-amplification of a partition $p'\in P$. Restriction of $\CC$ to partitions on its first level yields a category of partitions $\CC'\subset P$ such that $\CC=[\CC']^{(2)}$ -- hence $G$ is a Banica-Speicher quantum group by Proposition \ref{PropAmplif}.
\end{rem}

\subsection{Examples in the case $m=2$}

In this subsection, we restrict to the case $m=2$ and $n_1=n_2=n$ and we provide an incomplete list of categories $\CC\subset P_2^{(2)}$ of spatial pair partitions (all blocks are of size 2). In order to distinguish them, we introduce the following five sets.

\begin{defn}\label{DefMajorCateg}
We define the following subsets of $P_2^{(2)}$.
\begin{itemize}
\item[(a)] We let 
\[\CC_{\resplevels}:=\{\begin{pmatrix}p_2\\p_1\end{pmatrix}\;|\; p_1,p_2\in P_2\}\subset P_2^{(2)}\]
be the set of all spatial partitions respecting the levels.
\item[(b)] A spatial partition $p\in P_2^{(2)}$ is called \emph{level symmetric}, if it is symmetric when swapping the levels one and two. In other words, if two points $(x_1,y_1)$ and $(x_2,y_2)$ form a block of $p$, then also $(x_1,\bar y_1)$ and $(x_2,\bar y_2)$ form a block, where $\bar y:=\begin{cases} 1&\textnormal{ if } y=2\\ 2&\textnormal{ if } y=1\end{cases}$. We put
\[\CC_{\symm}:=\{p\in P_2^{(2)}\;|\; p \textnormal{ is level symmetric}\}\subset P_2^{(2)}.\]
\item[(c)] We let
\[\CC_{\nodiagonal}:=\{p\in P_2^{(2)}\;|\; \textnormal{no two points $(x,1)$ and $(y,2)$ with $x\neq y$ form a block}\}\subset P_2^{(2)}\]
be the set of all spatial partitions having no diagonal strings between the levels. We put
\[\CC_{\nodiagonal}^{\symm}:=\CC_{\nodiagonal}\cap\CC_{\symm}.\]
\item[(d)] We let
\[\CC_{\noviceversa}:=\{p\in P_2^{(2)}\;|\; \textnormal{no two points $(x,1)$ and $(x,2)$  form a block}\}\subset P_2^{(2)}\]
be the set of all spatial partitions having no geodesic strings between the levels. We put
\[\CC_{\noviceversa}^{\symm}:=\CC_{\noviceversa}\cap\CC_{\symm}.\]
\item[(e)] We let
\[\CC_{\even}:=\bigcup_{k+l=2n, \;n\in\N} P_2^{(2)}(k,l)\subset P_2^{(2)}\]
be the set of all spatial partitions whose number of blocks is even.
\end{itemize}
\end{defn}

\begin{rem}\label{RemMajorCateg}
We   have $\CC_{\resplevels}=\CC_{\nodiagonal}\cap\CC_{\noviceversa}$ and $\CC_{\resplevels}\subset\CC_{\even}$. Moreover, $\CC_{\noviceversa}^{\symm}\subset\CC_{\even}$.
\end{rem}

\begin{lem}\label{LemMajorCateg}
The sets $\CC_{\resplevels}$, $\CC_{\symm}$, $\CC_{\nodiagonal}^{\symm}$, $\CC_{\noviceversa}^{\symm}$ and $\CC_{\even}$ are categories of spatial partitions.
\end{lem}
\begin{proof}
We may use  Lemma \ref{LemPiGrading}(b) with $\pi=\singleton\otimes\singleton$ for the the set $\CC_{\resplevels}$. As for the others, one can directly verify stability under the category operations.
\end{proof}

 Recall that there are only three subcategories of $P_2$ in the case $m=1$, namely $NC_2$, $\langle\halflibpart\rangle$ and $P_2$ (see \cite{Web}). For $m=2$ we have many more.

\begin{thm}\label{ThmEx}
All of the following categories are subcategories of $P_2^{(2)}$. They are all distinct.
\begin{itemize}
\item[(a)] The amplifications $[NC_2]^{(2)}=\langle\emptyset\rangle$, $[\langle\halflibpart\rangle]^{(2)}=\langle\halflibpart^{(2)}\rangle$ and $[P_2]^{(2)}=\langle \CrossOneCrossTwopart\rangle$ (see Proposition \ref{PropAmplif}).
\item[(b)] The categories $\CC_1\times\CC_2$ with $\CC_i\in\{NC_2,\langle\halflibpart\rangle,P_2\}$  as in Section \ref{SectProducts}.
\item[(c)] The category $\langle\PaarPaarOneTwopart\rangle$.
\item[(d)] The category $\langle\CrossOneTwopart\rangle$.
\item[(e)] The category $\langle\CrossOneTwopart,\PaarPaarOneTwopart\rangle$.
\item[(f)] The category $\langle\PaarOneTwopart\rangle$.
\item[(g)] The category $\langle\CrossOneTwopart,\PaarOneTwopart\rangle$.
\item[(h)] The category generated by the following spatial partition.
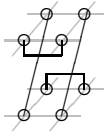
\setlength{\unitlength}{0.5cm}
\begin{center}
\begin{picture}(4,3)
   \savebox{\boxS}
   { \begin{picture}(2,2)
\put(-0.3,0.2){\line(0,1){0.4}}
\put(0.3,0.9){\line(0,1){0.4}}
      \end{picture}}    
   \savebox{\boxVierer}
   { \begin{picture}(2,2)
   \put(-0.3,0){\usebox{\boxGrLL}}
   \put(0.7,0){\usebox{\boxGrLL}}
      \end{picture}}
\put(0,0){\usebox{\boxVierer}}
\put(0,2){\usebox{\boxVierer}}
\put(0,0.2){\line(1,4){0.7}}
\put(1,0.2){\line(1,4){0.7}}
\put(0,2.2){\line(0,-1){0.4}}
\put(1,2.2){\line(0,-1){0.4}}
\put(0,1.8){\line(1,0){1}}
\put(0.6,0.9){\line(0,1){0.4}}
\put(1.6,0.9){\line(0,1){0.4}}
\put(0.6,1.3){\line(1,0){1}}
\end{picture}
\end{center}
\item[(i)] The category $P_2^{(2)}$ itself.
\end{itemize}
\end{thm}
\begin{proof}
We may distinguish the above categories using those of Lemma \ref{LemMajorCateg}: We have the following containments of categories. Observe that $p\in\CC$ if and only if $\langle p\rangle\subset \CC$, since $\langle p\rangle$ is the smallest category containing $p$.
\begin{center}
\begin{tabular}{c|ccccc}
&$\CC_{\resplevels}$ &$\CC_{\symm}$ &$\CC_{\nodiagonal}^{\symm}$ &$\CC_{\noviceversa}^{\symm}$ &$\CC_{\even}$\\
\hline
(a) & $\subset$ & $\subset$\\
(b) & $\subset$ & $\not\subset$\\
(c) & $\not\subset$ & $\subset$ & $\subset$ & $\not\subset$ & $\subset$\\
(d) & $\not\subset$ & $\subset$ & $\not\subset$ & $\subset$ & $\subset$\\
(e) & $\not\subset$ & $\subset$ & $\not\subset$ & $\not\subset$ & $\subset$\\
(f) & $\not\subset$ & $\subset$ & $\subset$ & $\not\subset$ & $\not\subset$\\
(g) & $\not\subset$ & $\subset$ & $\not\subset$ & $\not\subset$ & $\not\subset$\\
(h) & $\not\subset$ & $\not\subset$ & & & $\subset$ \\
(i) & $\not\subset$ & $\not\subset$ & & & $\not\subset$ \\
\end{tabular}
\end{center}
Hence, all of the categories (a) to (i) are distinct.
\end{proof}


It is very likely that the above list is not complete. However, we believe that (a) and (b) list all categories respecting the levels. The above categories are of interest since they correspond to quantizations of the orthogonal group $O_n$ in a way.
By Proposition  \ref{PropAmplif}, the amplifications $[NC_2]^{(2)}$,  $[\langle\halflibpart\rangle]^{(2)}$ and $[P_2]^{(2)}$ correspond to $O_{n^2}^+$, $O_{n^2}^*$ and $O_{n^2}$ respectively. By Theorem \ref{ThmProduct}, the categories $\CC_1\times\CC_2$ with $\CC_i\in\{NC_2,\langle\halflibpart\rangle,P_2\}$  correspond to glued tensor products of $O_n^+$, $O_n^*$ and $O_n$. As for determining the quantum groups corresponding to the categories (c-h) of Theorem \ref{ThmEx}, use the $C^*$-algebraic relations of Section \ref{SectRelations}. Note that the quantum groups of Example \ref{ExAmalgam} do not come into play here, since $\langle NC_2\times NC_2,p\rangle\neq\langle p\rangle$ in both cases due to $\CC_{\symm}$.

Concerning the quantum group $G$ corresponding to the category $P_2^{(2)}$, it is easy to check, like in Theorem \ref{PropMin}, that the elements
\[v_{(i_1,i_2)(j_1,j_2)}':=u_{i_1j_1}u_{i_2j_2}\in C(O_n)\]
satisfy all relations $R(p)$ for $p\in P_2^{(2)}$ (using Theorem \ref{PropPZwei}). However, this map is not surjective; in particular, $u_{ij}':=\sum_k v_{(i,k)(j,1)}\in C(G)$ does \emph{not} give rise to an orthogonal matrix (or equivalenty: $\sum_k u_{ij}u_{k1}\neq u_{ij}$ in $C(O_n)$). Thus, we have to leave the question open to which quantum group $P_2^{(2)}$ corresponds. 


\subsection{From quantum subgroups of $O_{n^2}^+$ to quantum subgroups of $O_n^+$}

Starting with a quantum subgroup $G$ of  $O_{n^2}^+$ we may associate a quantum subgroup $\mathring G$ of $O_n^+$ to it, under certain conditions.

\begin{defn}\label{DefRing}
Let $G\subset O_{n^2}^+$ be a compact matrix quantum group such that $C(G)$ is generated by $u_{(i,k)(j,l)}$, for $i,j,k,l=1,\ldots,n$. We put
\[\mathring u_{ij}:=\sum_k u_{(i,k)(j,1)}.\]
Let $C(\mathring G)\subset C(G)$ be the $C^*$-subalgebra of $C(G)$ generated by the elements $\mathring u_{ij}$ for $i,j\in\{1,\ldots,n\}$. 
\end{defn}

The elements $\mathring u_{ij}$ are self-adjoint. We now investigate, when $C(\mathring G)$ gives rise to a compact matrix quantum group $\mathring G\subset O_n^+$. We express the necessary condition in terms of $C^*$-algebraic relations $R(p)$ associated to partitions $p\in P^{(2)}$ as in Section \ref{SectRelations}. However, our next proposition does not only work for spatial partition quantum groups, it holds for general compact matrix quantum groups.

\begin{prop}\label{PropRing}
Suppose the relations $R(p)$ for $p=\IdOneSingSingTwopart$ are satisfied for the elements $u_{(i,k)(j,l)}\in C(G)$ and suppose $S_n\subset G\subset O_{n^2}^+$ (where $S_n\subset G$ is in the sense of Example \ref{ExCMQG}). Then:
\begin{itemize}
\item[(a)] We have, independently of the choice of $x$ and $y$,
\[\mathring u_{ij}=\sum_k u_{(i,k)(j,x)}=\sum_k u_{(i,y)(j,k)}.\]
\item[(b)] The map  $\Delta:C(G)\to C(G)\otimes C(G)$ restricts to $\Delta:C(\mathring G)\to  C(\mathring G)\otimes C(\mathring G)$ with $\Delta(\mathring u_{ij})=\sum_k \mathring u_{ik}\otimes\mathring u_{kj}$.
\item[(c)] The $C^*$-algebra $C(\mathring G)$ gives rise to a compact matrix quantum group  $\mathring G$ with
\[S_n\subset\mathring G\subset O_n^+.\]
\item[(d)] If in addition the relations $R(p)$ for $p\in\{\SingOneSingTwopart, \HalfThreepart\}$ are satisfied for the elements $u_{(i,k)(j,l)}\in C(G)$, then
\[S_n\subset\mathring G\subset S_n^+.\]
\end{itemize}
\end{prop}
\begin{proof}
(a) This is exactly what the relations $R(p)$ for $p=\IdOneSingSingTwopart$ are.

(b) We compute, using (a): 
\[\Delta(\mathring u_{ij})=\sum_l\Delta(u_{(i,l)(j,1)})=\sum_{l,k,m} u_{(i,l)(k,m)}\otimes u_{(k,m)(j,1)}=\sum_{k,m}\mathring u_{ik}\otimes u_{(k,m)(j,1)}=\sum_k \mathring u_{ik}\otimes\mathring u_{kj}.\]

(c) By (b), $\mathring G$ is a compact matrix quantum group. The matrix $\mathring u=(\mathring u_{ij})$ is orthogonal due to the following computation using (a) and $G\subset O_{n^2}^+$:
\[\sum_k\mathring u_{ik}\mathring u_{jk}=\sum_{k,m} \mathring u_{ik}u_{(j,1)(k,m)}=\sum_{k,l,m} u_{(i,l)(k,m)}u_{(j,1)(k,m)}=\sum_l\delta_{ij}\delta_{l1}=\delta_{ij}.\]
Similarly $\sum_k\mathring u_{ki}\mathring u_{kj}=\delta_{ij}$. Hence, $\mathring G\subset O_n^+$. As for proving $S_n\subset \mathring G$, note that by assumption we have a $^*$-homomorphism $\phi: C(G)\to C(S_n)$ mapping $u_{(i,k)(j,l)}$ to $v_{ij}v_{kl}$, where we denote the generators of $C(S_n)$ by $v_{ij}$. Thus, $\phi(\mathring u_{ij})=v_{ij}$ which proves $S_n\subset \mathring G$.

(d) All we have to check is that the elements $\mathring u_{ij}$ satisfy $\mathring u_{ij}^2=\mathring u_{ij}$ and $\sum_l\mathring u_{il}=\sum_l\mathring u_{lj}=1$. This  follows directly from (a) and the relations $R(p)$ which we list below.
\begin{align*}
&R(p) \textnormal{ for } p=\SingOneSingTwopart:
&&\sum_{g_1,g_2} u_{(b_1,b_2)(g_1,g_2)}
=\sum_{g_1,g_2} u_{(g_1,g_2)(b_1,b_2)}
=1;
\\
&R(p) \textnormal{ for } p=\HalfThreepart:
&&\sum_g u_{(b_1,b_2)(i_1,i_2)}u_{(b_1,g)(i_3,i_4)}=\delta_{i_1i_3}u_{(b_1,b_2)(i_1,i_2)}.
\end{align*}
\end{proof}

\begin{rem}
We may also define $\mathring G$ via $\mathring u_{ij}:=\sum_k u_{(k,i)(1,j)}$ and require the relations $R(p)$ with $p=\SingSingOneIdTwopart$ in Proposition \ref{PropRing}; this will yield an analogue result.
\end{rem}

We conclude that we may produce quantum groups $\mathring G$ which are intermediate between $S_n$ and $O_n^+$, just like Banica-Speicher quantum groups. However, it is not clear for the moment whether or not they yield quantum groups which are not Banica-Speicher quantum groups (the question of finding non-easy quantum groups).

\section{Outlook: the unitary case}
\label{SectOutlook}

\subsection{Categories of colored spatial partitions}
In the spirit of \cite{TWcomb, TW}, it is clear how to define unitary spatial partition quantum groups $G\subset U_{n_1\cdots n_m}^+$. Namely, we color each point of the first level of a spatial partition $p\in \Pm(k,l)$  either in white ($\circ$) or in black ($\bullet$); we then copy this color pattern to all other levels. In other words, we do not color all points arbitrarily -- the colors of all points $(x,y)\in\{1,\ldots,k+l\}\times\{1,\ldots,m\}$ for a fixed $x$ coincide. A category of colored spatial partitions is defined as in Definition \ref{DefmCateg} replacing $\paarpart^{(m)}$ and $\idpart^{(m)}$ by $\paarpartwb^{(m)}$, $\paarpartbw^{(m)}$, $\idpartww^{(m)}$ and $\idpartbb^{(m)}$. Remark that any uncolored category can be seen as a colored category by admitting any color pattern on the partitions of the uncolored category.

We then associate linear maps $S_p$ to such a colored partition $p$ exactly as in Section \ref{SectLinMMaps} -- the colorization of $p$ does not play any role for this definition. However, for the interpretation of $S_p$ as an intertwiner we \emph{do} need the colors of the points: If the color pattern of the upper first level of $p$ is the word $w\in\{\circ,\bullet\}^k$ whereas $s\in\{\circ,\bullet\}^l$ colors the lower first level, the map $S_p$ is supposed to be an intertwiner of the representations $u^w$ and $u^s$. Here, $u^w$ and $u^s$ are tensor products of $u$ and $\bar u$ according to the identifications $u^\circ=u$ and $u^\bullet=\bar u$. Using in a similar way the Tannaka-Krein duality, we can produce a unitary quantum group (we do not assume anymore that $u=\bar{u}$) from any category of colored spatial partitions.

\subsection{Noncrossing product of categories}

In analogy to Section \ref{SectProducts}, we may now define new categories from old ones. We first need a notion of non-crossing colored spatial partitions. 
For this purpose, we extend the isomorphism of Remark \ref{RemPmAndP} from white points to white and black points: 
For any $m\in\N$, $k,l\in\N_0$ and a fixed color pattern on $\{1,\ldots,k,k+1,\ldots,k+l\}\times\{1,\ldots,m\}$, the sets 
\[A:=\{1,\ldots,km,km+1,\ldots,km+lm\}\]
and
\[B:=\{1,\ldots,k,k+1,\ldots,k+l\}\times\{1,\ldots,m\}\]
are in bijective correspondence by identifying a point $(x-1)m+y\in A$, $1\leq x\leq k+l$, $1\leq y\leq m$ with the point $(x,y)\in B$ if $(x,y)$ is white and by identifying a point $xm+1-y\in A$, $1\leq x\leq k+l$, $1\leq y\leq m$ with the point $(x,y)\in B$ if $(x,y)$ is black. This reverse order on black points reflects the identity $\overline{u\otimes v}=\overline{v}\otimes \overline{u}$.

We can define two different products using this isomorphism. The definition of the noncrossing product $\CC_1\times_{\nc}\CC_2$ and of the free product $\CC_1\ast\CC_2$ of categories $\CC_1\subset P$ and $\CC_2\subset P$ of colored partitions follow Definition \ref{DefKronecker} with additional conditions of being noncrossing:
\begin{itemize}
\item We place a partition $p$ from $\CC_1$ on level one, and a partition $q$ from $\CC_2$ on level two and consider the resulting partition $\binom{q}{p}$.
\item For the definition of $\CC_1\times_{\nc}\CC_2$, we require that the partitions $p$ and $q$ do not cross each other under the above isomorphism. More precisely, $\binom{q}{p}$ is in $\CC_1\times_{\nc}\CC_2$ whenever there exists a partition $r\leq \binom{q}{p}$ which respects the levels and which is noncrossing under the above isomorphism.
\item For the definition of $\CC_1\ast\CC_2$, we require in addition that for each block of $r$ on the first level, the restriction of $p$ to this block is in $\CC_1$, and for each block of $r$ on the second level, the restriction of $q$ to this block is in $\CC_2$.
\end{itemize}
As a consequence, $\CC_1\ast\CC_2\subset \CC_1\times_{\nc}\CC_2$.
For example, putting $\paarpartwb$ on both level gives us the noncrossing partition $\{\{1,4\},\{2,3\}\}$, meaning that $\paarpartwb^{(2)}$ is an element of the noncrossing product $\CC_1\times_{\nc}\CC_2$. In addition, each restriction is $\paarpartwb$, which is in every category of colored partitions, meaning that $\paarpartwb^{(2)}\in \CC_1\ast\CC_2$. Similarly, we can verify that $\paarpartbw^{(2)}$, $\idpartww^{(2)}$ and $\idpartbb^{(2)}$ are in the free product $\CC_1\ast\CC_2$. Since the conditions are maintained under the category operations, we deduce that $\CC_1\times_{\nc}\CC_2\subset P^{(2)}$ and $\CC_1\ast\CC_2\subset P^{(2)}$ are categories of colored spatial partitions on two levels. Remark that for some categories (including $P_2$ and $NC_2$), we have $\CC_1\times_{\nc}\CC_2=\CC_1\ast\CC_2$.

Theorem~\ref{ThmFreeProduct} is the unitary version of Theorem \ref{ThmProduct}, where we used the glued free product of \cite[Def. 6.4]{TW}: if $(G,u)$ and $(H,v)$ are two compact matrix quantum groups with $u=(u_{ij})_{i,j=1,\ldots,n}$ and $v=(v_{kl})_{k,l=1,\ldots,m}$, the \emph{glued direct product} $G\tilde\ast H$ of $G$ and $H$ is given by the $C^*$-subalgebra
\[C^*(u_{ij}v_{kl}\;|\; i,j=1,\ldots,n\textnormal{ and } k,l=1,\ldots m)\subset C(G)\ast C(H).\]

\begin{thm}\label{ThmFreeProduct}
Let $(G_i,u_i)\subset O_n^+$ be Banica-Speicher quantum groups with categories $\CC_i\subset P$ for $i=1,2$. The spatial partition quantum group associated to the category $\CC_1* \CC_2$ is $G_1\tilde{*} G_2\subset U_{n^2}^+$.
\end{thm}
\begin{proof}The intertwiners between tensor products of the representations $u_1,\bar{u}_1,u_2$ and $\bar{u}_2$ of $G_1* G_2$ are explicitely given by \cite[Proposition 2.15]{L}: they are linear combinations of compositions of morphisms of the type $\id\otimes R \otimes \id$ where $R$ is either an intertwiner between tensor products of the representations $u_1,\bar{u}_1$ of $G_1$ or an intertwiner between tensor products of the representations $u_2,\bar{u}_2$ of $G_2$.

Let us describe this set of intertwiners in a different way. Two tensor products of the representations $u_1,\bar{u}_1,u_2$ and $\bar{u}_2$ of length $k$ and $l$ can be seen as a decoration of $k$ upper points and $l$ lower points by $u_1,\bar{u}_1,u_2$ and $\bar{u}_2$. Given such a decoration, we can consider a partition $p$ of $k+l$ points which does not connect the points decorated by $u_1,\bar{u}_1$ with the points decorated by $u_2,\bar{u}_2$ such that:
\begin{itemize}
\item there exists a noncrossing partition $r$ which is coarser than $p$ and which does not connect the points decorated by $u_1,\bar{u}_1$ with the points decorated by $u_2,\bar{u}_2$,
\item the restriction of $p$ to each block of $r$ decorated by $u_1,\bar{u}_1$ is in $\CC_1$, and the restriction of $p$ to each block of $r$ decorated by $u_2,\bar{u}_2$ is in $\CC_2$.
\end{itemize}
Let us call such a partition an admissible partition. They form a category of colored partitions. We claim that the intertwiners between tensor products of the representations $u_1,\bar{u}_1,u_2$ and $\bar{u}_2$ of $G_1* G_2$ are exactly given by the linear combinations of the morphisms $T_p$ for $p$ an admissible partition in the sense above.

Let us briefly sketch the proof. On one hand, if $p$ is in $\CC_1$ or $\CC_2$, the morphism $\id\otimes T_p \otimes \id$ can be written as $T_{\idpart^{\otimes a}\otimes p\otimes\idpart^{\otimes b}}$ with the admissible partition $\idpart^{\otimes a}\otimes p\otimes\idpart^{\otimes b}$. Taking the closure by linear combination and composition gives us that every intertwiner between tensor products of the representations $u_1,\bar{u}_1,u_2$ and $\bar{u}_2$ is given by the linear combinations of the morphisms $T_p$ for $p$ an admissible partition. Conversely, if $p$ is an admissible partition, there exists a noncrossing partition $r$ which is as described above. There exists at least one interval block in $r$. Doing some rotation if necessary, we can assume that this block is supported on consecutive points on the upper left corner. Because $p$ is admissible, the restriction of $p$ to this block of $r$ is a partition $p_1$ of $\CC_1$ or $\CC_2$, and we can decompose $p=p_1\otimes p_2$ with $p_2$ admissible. But $T_{p_1}$ is an intertwiner between tensor products of the representations $u_1,\bar{u}_1,u_2$ and $\bar{u}_2$ of $G_1\tilde{*} G_2$. We conclude by induction on the number of blocks of the admissible partitions.

Thus we can describe the set of intertwiners between tensor products of the representations $u_1\otimes u_2$ and $\overline{u_1\otimes u_2}=\bar{u}_2\otimes \bar{u}_1$ of $G_1* G_2$, or equivalently between tensor products of the fundamental representation of $G_1\tilde{*} G_2$ and its adjoint, as linear combination of $T_p$ with $p$ an admissible partition. Using the isomorphism described at the beginning of the section, we see that it coincides exactly with
$$\lspan\{S_p\;|\;p\in \CC_1* \CC_2\}$$
as wanted.
\end{proof}

We observe that while $S_n^+\not\subset O_{n^2}^+$, we have $S_n^+\subset U_{n^2}^+$ in the sense of Example \ref{ExCMQG}(c) and Theorem \ref{PropMin}. 
It would be interesting to classify all spatial partition quantum groups $S_n^+\subset G\subset U_{n^2}^+$ since these are the ones relevant for free probability in the sense of \cite{VSW, SpK, BCS}.

\section*{Acknowledgements}

The first author was funded by the ERC Advanced Grant NCDFP held by Roland Speicher; the second author was partially funded by the same grant.

\bibliographystyle{alpha}
\nocite{*}
\bibliography{BibSpatialQG}

\end{document}